\documentclass{tran-l}

\usepackage{amssymb,amsfonts,amsthm}
\usepackage{marginnote}
\usepackage{comment}
\usepackage{enumitem}
\usepackage[dvipsnames]{color}

\usepackage[colorlinks=true, pdfstartview=FitV, linkcolor=black, citecolor=blue, urlcolor=blue]{hyperref}

\def\Xint#1{\mathchoice
{\XXint\displaystyle\textstyle{#1}}%
{\XXint\textstyle\scriptstyle{#1}}%
{\XXint\scriptstyle\scriptscriptstyle{#1}}%
{\XXint\scriptscriptstyle%
\scriptscriptstyle{#1}}%
\!\int}
\def\XXint#1#2#3{{\setbox0=\hbox{$#1{#2#3}{%
\int}$ }
\vcenter{\hbox{$#2#3$ }}\kern-.6\wd0}}
\def\barint{\, \Xint -} 
\def\bariint{\barint_{} \kern-.4em \barint}
\def\bariiint{\bariint_{} \kern-.4em \barint}
\renewcommand{\iint}{\int_{}\kern-.34em \int} 
\renewcommand{\iiint}{\iint_{}\kern-.34em \int} 

\DeclareMathAlphabet{\mathcal}{OMS}{cmsy}{m}{n}

\newcommand{\R}{\mathbb{R}}

\newcommand{\N}{\mathbb{N}}
\newcommand{\Z}{\mathbb{Z}}

\newcommand{\supp}{\mathop{\mathrm{supp}}}

\newcommand{\de}{\delta}

\newcommand{\ga}{{\gamma}}

\newcommand{\lam}{\lambda}
\newcommand{\La}{\Lambda}
\newcommand{\Om}{{\Omega}}

\newcommand{\td}{\tilde}

\renewcommand{\th}{\theta}

\newcommand{\nb}{\nabla}
\newcommand{\p}{\partial}
\newcommand{\la}{\langle}
\newcommand{\ra}{\rangle}
\newcommand{\les}{\lesssim}
\newcommand{\ges}{\gtrsim}
\newcommand{\norm}[1]{\lVert #1 \rVert}
\renewcommand{\:}{\colon}

\newcommand{\wto}{\rightharpoonup} 
\newcommand{\wstar}{\overset{\ast}{\rightharpoonup}}
\newcommand{\into}{\hookrightarrow}

\newcommand{\uloc}{\mathrm{uloc}}

\newcommand{\loc}{{\rm loc}}

\newcommand{\lin}{{\rm lin}}

\let\div\relax
\DeclareMathOperator{\div}{div}
\DeclareMathOperator{\pv}{pv}

\let\tilde\relas
\newcommand{\tilde}[1]{\widetilde{#1}}
\DeclareMathOperator*{\esssup}{ess\,sup}

\newcommand{\I}{{\infty}}  

\newcommand{\BMO}{{\rm BMO}}

\newcommand{\osc}{{\rm osc}}

\newcommand{\blue}[1]{{\color{blue} \bf #1}}

\newcommand{\cmz}[1]{{\color{magenta} #1}}

\newcommand{\EQ}[1]{\begin{equation}\begin{split} #1 \end{split}\end{equation}}
\newcommand{\EQN}[1]{\begin{equation*}\begin{split} #1 \end{split}\end{equation*}} 

\newtheorem{theorem}{Theorem}[section]
\newtheorem{lemma}[theorem]{Lemma}

\newtheorem{corollary}[theorem]{Corollary}
\newtheorem{proposition}[theorem]{Proposition}

\theoremstyle{definition}

\theoremstyle{remark}
\newtheorem{remark}[theorem]{Remark}

\numberwithin{equation}{section}

\begin{document}

\title[Non-decaying solutions to critical SQG]{Non-decaying solutions to the critical surface quasi-geostrophic equations with symmetries}
\author{Dallas Albritton}
\address{Courant Institute of Mathematical Sciences, New York University, New York, NY 10012}
\email{daa399@cims.nyu.edu}

\author{Zachary Bradshaw}
\address{Department of Mathematics, University of Arkansas, Fayetteville, AR 72701}
\email{zb002@uark.edu}


\date{\today}

\begin{abstract}
We develop a theory of self-similar solutions to the critical surface quasi-geostrophic equations. We construct self-similar solutions for arbitrarily large data in various regularity classes and demonstrate, in the small data regime, uniqueness and global asymptotic stability. These solutions are non-decaying as $|x| \to +\infty$, which leads to ambiguity in the velocity $\vec{R}^\perp \theta$. This ambiguity is corrected by imposing $m$-fold rotational symmetry.  The self-similar solutions exhibited here lie just beyond the known well-posedness theory and are expected to shed light on potential non-uniqueness, due to symmetry-breaking bifurcations, in analogy with work~\cite{jiasverakillposed,guillodsverak} on the Navier-Stokes equations.
\end{abstract}

\maketitle

\setcounter{tocdepth}{1}
\tableofcontents

\section{Introduction}

 The uniqueness of Leray's weak solutions~\cite{leray} of the three-dimensional Navier-Stokes equations
\begin{equation}
\label{eq:NS}
\tag{NS}
\begin{aligned}
	\p_t u + u \cdot \nabla u - \Delta u + \nabla p = 0, \quad \div u = 0,
	\end{aligned}
\end{equation}
is a well known open problem. In~\cite{jiasverakselfsim,jiasverakillposed}, Jia and {\v S}ver{\'a}k proposed a non-uniqueness scenario based on \emph{forward self-similar solutions} $u \: \R^3 \times \R_+ \to \R^3$, which are invariant under the scaling symmetry
\begin{equation}
\label{eq:NSscaling}
	u \to \lambda u(\lambda x,\lambda^2 t), \quad p \to \lambda^2 p(\lambda x,\lambda^2 t).
\end{equation}
These solutions, which lie just beyond the known perturbation theory, correspond to steady states of the Navier-Stokes equations in similarity variables $u(x,t) = U(y)/\sqrt{t}$, $y = x/\sqrt{t}$:  
\begin{equation}
	\label{eq:NSinselfsim}
	- \Delta U + U \cdot \nabla U  -(1 + y \cdot \nabla) U/2 + \nabla P = 0, \quad \div U = 0.
\end{equation}
Based on the analogy between~\eqref{eq:NSinselfsim} and the steady Navier-Stokes equations, which naturally exhibit bifurcations, Jia and {\v S}ver{\'a}k conjectured that similar bifurcations occur within the class of steady solutions of~\eqref{eq:NSinselfsim} and cause a loss of uniqueness in the Cauchy problem for~\eqref{eq:NS}. Convincing numerical evidence for these bifurcations was recently discovered by Guillod and {\v S}ver{\'a}k~\cite{guillodsverak}.  \\


With the above paradigm in mind, we explore connections between self-similar solutions and potential non-uniqueness in the \emph{critical surface quasi-geostrophic equations}:
\begin{equation}
	\label{eq:sqg}
	\tag{SQG}
\partial_t \theta + \vec{R}^\perp \theta \cdot\nabla \theta + \Lambda \theta  = 0,
\end{equation}
where $\theta \: \R^2 \times \R_+ \to \R$ is the surface buoyancy, $\Lambda = (- \Delta)^{1/2}$ is the half-Laplacian, and $\vec{R}$ is the vector of Riesz transforms. 
The above equation arises in geophysical fluid dynamics as a model of the large-scale motion of the ocean when the quasi-geostrophic potential vorticity is uniform. The dissipation term $\Lambda \theta$ captures the Ekman friction (see~\cite{lapeyre2017surface} for a survey on SQG from a physics perspective). This model and its cousins also arise in attempts to understand potential singularity formation in the Navier-Stokes and Euler equations~\cite{ConstantinMajdaTabak}, since $\omega = \nabla^\perp \theta$ and $u = \vec{R}^\perp \theta$ satisfy
\begin{equation}
	\p_t \omega + u \cdot \nabla \omega - \omega \cdot \nabla u + \Lambda \omega = 0,
\end{equation}
which is analogous to the vector transport equation satisfied by the vorticity. 

Unlike~\eqref{eq:NS} in dimension three, \eqref{eq:sqg} is \emph{critical} in the sense that its strongest known monotone quantity, the $L^\infty$ norm, is invariant under the scaling symmetry
\begin{equation}
	\label{eq:SQGscaling}
	\theta \to \theta(\lambda x,\lambda t).
\end{equation}
One of our motivations is to better understand how criticality impacts the non-uniqueness program.
By now, it is well understood that solutions are globally well-posed in the smooth category, with many proofs, including~\cite{KiselevNazarovVolbergInventiones2007} (`moduli of continuity'),~\cite{CV} (inspired by De Giorgi),~\cite{VariationsThemeCaffVass2009}, ~\cite{ConstVicolGAFA} (`nonlinear maximum principle'),~\cite{MaekawaMiura-drift} (inspired by Nash), and \cite{ConstTarfVicolCMP}. However, self-similar (i.e., $0$-homogeneous) initial data is discontinuous at the origin (unless $\theta_0 \equiv \text{const.})$ and does not belong to the above well-posedness theory. 
Our second motivation is to better understand difficulties concerning self-similar solutions of the two-dimensional Navier-Stokes system, whose vorticity formulation is
\begin{equation}
	\label{eq:NS2d}
	\tag{NS2}
	\p_t \omega + u \cdot \nabla \omega - \Delta \omega = 0, \quad u = \nabla^\perp \psi, \quad \Delta \psi = \omega.
\end{equation}
While equations~\eqref{eq:sqg} and~\eqref{eq:NS2d} are dissipative active scalars and share many features, the existence of large self-similar solutions of~\eqref{eq:NS2d} appears to be more difficult. We will discuss the differences in detail below.

\subsection{Main results}

In developing a theory of non-decaying solutions of~\eqref{eq:sqg}, we immediately encounter an essential issue at the level of making sense of the equations: The Riesz transform of a non-decaying function is generally only well-defined \emph{up to constants} (even if the function is smooth). One may view this as an issue of gauge invariance due to the generalized Galilean boosts $x \to x + y(t)$. When $\theta \in L^2$, there is a single representative of $\vec{R}^\perp \theta$ belonging to $L^2$, and this choice fixes the gauge. For non-decaying functions, there may no longer be a natural choice unless one imposes additional structure, such as periodicity. Our approach is to impose symmetries that automatically determine a representative of the Riesz transform and, hence, fix the gauge.

\subsubsection*{Symmetries} Let $g \in O(1) = \{ g \in \R^{2\times2} : gg^T = 1\}$ act on functions $\theta \: \R^2 \to \R$ and vector fields $\vec{v} \: \R^2 \to \R^2$ according to the transformations
\begin{equation}
	 (g \cdot \theta)(x) = (\det g) \theta(g^{-1} x), \quad (g \cdot \vec{v})(x) = (\det g) g \vec{v}(g^{-1} x)
\end{equation}
where we omit $\cdot$ when $g$ acts on $\R^2$. In this convention, reflections act on $\theta$ by odd reflection. It is simple to formally verify that if $\theta$ is a solution of~\eqref{eq:sqg}, then $g \cdot \theta$ is also a solution of~\eqref{eq:sqg}. 

We consider the symmetry groups
\begin{equation}
	\label{eq:symmetrygroups}
	G = \la R_{\varphi_0} \ra, \la R_{\varphi_0}, S_{y=0} \ra, SO(1)
\end{equation}
where $\varphi_0 \in 2\pi( \mathbb{Q} \setminus \Z )$, $R_{\varphi}$ denotes counter-clockwise rotation by angle $\varphi$, $S_{y=0}$ denotes reflection across the $x$-axis, and $\la \cdot \ra$ denotes generation by. The above groups correspond to $m$-fold rotational symmetry, $m$-fold rotational and odd-in-period symmetry, and radial symmetry. They are precisely those proper closed subgroups of $O(1)$ that do not stabilize any non-zero vector in $\R^2$. If $\theta(\cdot,t) \in L^\infty$ is invariant under $r = R_{\varphi_0}$, then, for each ball $B_N$, $N > 0$,
 \begin{equation}
 \label{eq:quickrotationcomputation}
	r \int_{B_{N}} (\vec{R}^\perp \theta)(x,t) \, dx = \int_{B_{N}} r \cdot \vec{R}^\perp \theta \, dx = \int_{B_{N}} \vec{R}^\perp (r \cdot \theta) \, dx = \int_{B_{N}} \vec{R}^\perp \theta \, dx.
 \end{equation}
 Hence, each term in~\eqref{eq:quickrotationcomputation} vanishes. The symmetry assumption is considerably stronger than fixing an arbitrary gauge because it automatically controls the mean drift at every scale. \\ 

In the sequel, $G$ is a fixed symmetry group from~\eqref{eq:symmetrygroups}. For simplicity, we often write \emph{symmetric}, rather than $G$-symmetric, to mean invariant under $G$.

We consider a generalization of self-similarity known as \emph{discrete self-similarity}, i.e.,  invariance under~\eqref{eq:SQGscaling} for a specific scaling factor $\lambda > 1$. We write $\lambda$-DSS to mean $\lambda$-discretely self-similar. Plainly, every self-similar solution is $\lambda$-DSS for all $\lambda>1$.\\

We now present our main results.

\begin{theorem}[Bounded solutions]
\label{thm:boundedsolutions}
Let $M > 0$ and $\theta_0 \in L^\infty(\R^2)$ be symmetric with $\norm{\theta_0}_{L^\infty(\R^2)} \leq M$.

(Existence) There exists a bounded, smooth, symmetric solution of~\eqref{eq:sqg} on $\R^2 \times (0,+\infty)$ satisfying $\theta(\cdot,t) \wstar \theta_0$ in $L^\infty(\R^2)$ as $t \to 0^+$.

(Self-similarity) The solution $\theta$ may be chosen to satisfy the following property: If $\lambda > 1$ and $\theta_0$ is $\lambda$-DSS, then $\theta$ is $\lambda$-DSS. In particular, $\theta$ may be chosen to be self-similar if $\theta_0$ is self-similar.

(\emph{A priori} smoothing estimates) Any bounded, smooth, symmetric solution with initial data $\theta_0$ satisfies the maximum principle
\begin{equation}
	\label{eq:maxprinciple}
	\norm{\theta(\cdot,t)}_{L^\infty(\R^2)} \leq \norm{\theta_0}_{L^\infty(\R^2)},
\end{equation}
the spatial regularity estimates
\begin{equation}
	\sup_{t > 0} t^\ell \norm{\nabla^\ell_x \theta(\cdot,t)}_{L^\infty(\R^2)} \leq C(\ell,M),
\end{equation}
for all integers $\ell \geq 0$, and the spacetime regularity estimates
\begin{equation}
	\label{eq:spacetimeregests}
	\norm{\p_t^k \nabla^\ell_x \theta(\cdot,t)}_{L^\infty_{t,x}(K)} \leq C(k,\ell,M,K)
\end{equation}
 for all compact $K = K_1 \times K_2 \subset \R^2 \times \R_+$ and integers $k, \ell \geq 0$.

(Weak-$\ast$ stability) If $\theta_0^{(k)} \wstar \theta_0$ in $L^\infty(\R^2)$ and $\theta^{(k)}$ are corresponding bounded, smooth, symmetric solutions, then there exists a subsequence such that $\theta^{(k)} \to \theta$ in $C^\infty_\loc(\R^2 \times \R_+)$, where $\theta$ is a bounded, smooth, symmetric solution with initial data $\theta_0$.
\end{theorem}

The maximum principle~\eqref{eq:maxprinciple} is not as obvious as expected, since \emph{a priori} the $L^\infty$ norm of a bounded, smooth solution could `jump' at the initial time (see Lemma~\ref{lem:Linftytheory}).

The specific form of the time regularity estimates in~\eqref{eq:spacetimeregests} is because, although $\vec{R}^\perp \theta$ is well defined, it is not controlled in $L^\infty$. We address this, and uniqueness, in Theorem~\ref{thm:moreregularsolutions}.


Let $\alpha \in (0,1)$. For each $f \in L^1_\loc(\R^2)$, we define the semi-norm
\begin{equation}
	\label{eq:ydotdef}
	\norm{f}_{\dot Y^\alpha} = \sup_{r > 0} r^\alpha [f]_{C^\alpha(\R^2 \setminus B_r)}.
\end{equation}
Let $\theta_0 \in L^\infty$ be symmetric, with $\vec{R} \theta_0 \in L^\infty$ (equivalent to $\vec{R}^\perp \theta_0 \in L^\infty$), $\nabla \theta_0 \in L^{2,\infty}$, and
\begin{equation}
	\label{eq:ybbdef}
	\norm{\theta_0}_{\mathbb{Y}^\alpha} = \norm{\theta_0}_{L^\infty}  + \norm{\theta_0}_{\dot Y^\alpha}  + \norm{\vec{R} \theta_0}_{L^\infty} + \norm{\nabla \theta_0}_{L^{2,\infty}} \leq M.
\end{equation}
This assumption is satisfied by $\lambda$-DSS $\theta_0$ locally Lipschitz in $\R^2 \setminus \{ 0 \}$.

\begin{theorem}[Regular solutions]
\label{thm:moreregularsolutions}
Let $\theta_0$ be as above and $\theta$ be as in Theorem~\ref{thm:boundedsolutions}.

(\emph{A priori} estimates)
We decompose $\theta$ as
\begin{equation}
	\theta = e^{-t\Lambda} \theta_0 + \psi.
\end{equation}
Then, for all $p \in (2,+\infty)$, we have
\begin{equation}
	 \sup_{t > 0}  t^{- \frac{2}{p}} \norm{\psi(\cdot,t)}_{L^p(\R^2)} \leq C(p,\alpha,M).
\end{equation}
Moreover, $\theta$ satisfies the spacetime regularity estimates
\begin{equation}
	\sup_{t > 0} t^{k+\ell} \norm{\p_t^k \nabla^\ell_x \theta(\cdot,t)}_{L^\infty(\R^2)} \leq C(k,\ell,\alpha,M)
\end{equation}
for all integers $k,\ell \geq 0$.

(Uniqueness) If $M \ll_\alpha 1$, then $\theta$ is unique in the class of bounded, smooth, symmetric solutions.
\end{theorem}

Classically, self-similar solutions describe the long-time behavior of certain evolutionary PDEs with a scaling symmetry. Among the numerous examples are the heat equation (heat kernel), porous medium equation (Barenblatt solutions), viscous Burgers equation (diffusion waves), and~\eqref{eq:NS2d} (Lamb-Oseen vortices, see~\cite{GallayWayne}). Small self-similar solutions play a similar role in~\eqref{eq:sqg}.

To state Theorem~\ref{thm:condasstab}, we reintroduce $G$ into the notation (see Remark~\ref{rmk:bifrmk}).
Let $\theta_0^{\rm ss} \in L^\infty$ be $G$-symmetric and self-similar. Let $\theta^{\rm ss}$ be a bounded, smooth, $G$-symmetric, self-similar solution with initial data $\theta_0^{\rm ss}$. Let $\theta_0 \in L^\infty$ satisfying
\begin{equation}
	\label{eq:initialdataassumption}
	\norm{\theta_0 - \theta_0^{\rm ss}}_{L^\infty(\R^2 \setminus B_R)} \to 0 \text{ as } R \to +\infty.
\end{equation}
Let $\theta$ be a bounded, smooth, $G$-symmetric solution with initial data $\theta_0$.

\begin{theorem}[Conditional asymptotic stability]
	\label{thm:condasstab}
Suppose that $\theta^{\rm ss}$ is the unique solution in the class of bounded, smooth, $G$-symmetric solutions with initial data $\theta_0^{\rm ss}$. Then, for all $R > 0$, we have
\begin{equation}
	\label{eq:asymptoticstability}
	\norm{\theta(\cdot,t) - \theta^{\rm ss}(\cdot,t)}_{L^\infty_x(B(Rt))} \to 0 \text{ as } t \to +\infty.
\end{equation}
In particular, Theorem~\ref{thm:moreregularsolutions} guarantees asymptotic stability when $\norm{\theta_0^{\rm ss}}_{{\rm Lip}(S^1)} \ll 1$.
\end{theorem}

Clearly, uniqueness of $\theta^{\rm ss}$ within the self-similar class is necessary for its asymptotic stability to hold. The above notion of convergence is equivalent to $L^\infty_\loc(\R^2)$ convergence of the profile $\Theta$ in similarity variables: $\Theta(y,s) = \theta(x,t)$, where $s = \log t$, $y = x/t$. An analogous result holds in the discretely self-similar case.

\begin{proof}[Proof of Theorem~\ref{thm:condasstab}]
Let $(t_k)_{k \in \N} \subset \R_+$ with $1 \leq t_k \to +\infty$ as $k \to +\infty$. Consider the sequence $(\theta^{(k)})_{k \in \N}$ of rescaled solutions
\begin{equation}
	\label{eq:therescaling}
	\theta^{(k)}(x,t) = \theta(t_k x, t_k t)
\end{equation}
with initial data $\theta_0^{(k)}(x) = \theta_0(t_k x)$. Notice that $\theta_0^{(k)} \wstar \theta_0^{\rm ss}$ in $L^\infty(\R^2)$ as $k \to +\infty$. By the weak-$\ast$ stability  property in Theorem~\ref{thm:boundedsolutions}, there exists a subsequence (still indexed by $k$) and a bounded, smooth, $G$-symmetric solution $\bar{\theta}$ such that
\begin{equation}
	\theta^{(k)} \to \bar{\theta} \text{ in } C^\infty_\loc(\R^2 \times \R_+) \text{ as } k \to +\infty
\end{equation}
and $\bar{\theta}(\cdot,0) = \theta_0^{\rm ss}$. By the uniqueness assumption, we have $\bar{\theta} \equiv \theta^{\rm ss}$. Since $(t_k)_{k \in \N}$ was arbitrary, unfolding the rescaling~\eqref{eq:therescaling} gives the asymptotic stability~\eqref{eq:asymptoticstability}.
\end{proof}

We expect that the convergence in Theorem~\ref{thm:condasstab} may be refined when $\theta_0^{\rm ss}$ is small. For example, we expect that bounded, (initially) compactly supported, $G$-symmetric perturbations converge in $L^\infty$ to the unique small self-similar solution at the rate $O(t^{-2})$. The situation when $\theta_0^{\rm ss}$ is not small but its solution is unique may be more subtle. We leave this and other refinements for future work.

\begin{remark}
\label{rmk:bifrmk}
Crucially, the uniqueness condition in Theorem~\ref{thm:condasstab} is allowed to depend on the group $G$. Consider the following scenario, which is inspired by the work of Jia, {\v S}ver{\'a}k, and Guillod~\cite{jiasverakillposed,guillodsverak}. Let $\bar{G} \subset G$ be a proper subgroup. For example, $\bar{G}$ contains only rotations while $G$ also contains a reflection. Consider the $1$-parameter family of initial data $A \theta_0^{\rm ss}$, $A > 0$, which presumably `draws' a branch of self-similar solutions in phase space. The solutions are unique when $A \ll 1$. Plausibly, as $A$ approaches a threshold value $A = A_{\rm crit}$, one may find a bifurcation that breaks the reflection symmetry and, presumably, exchanges stability. In this case, (i) uniqueness would be lost within the $\bar{G}$-symmetric class but retained within the $G$-symmetric class, and (ii) asymptotic stability would continue to hold among $G$-symmetric solutions.
\end{remark}



Our last result concerns the global existence of symmetric solutions with non-decaying \textit{unbounded} initial data. Since we want our theory to include scaling-invariant solutions, we cannot work with the uniformly local space $L^p_\uloc$ common in the Navier-Stokes literature~\cite{lemarie2002}. Instead, we work with a weighted $L^p$-based space $X_p^{R_0}$, $R_0 > 0$, with norm
\begin{equation}
	\norm{f}_{X_p^{R_0}} :=\sup_{R\geq R_0} \left( \barint_{B_R} |f|^p \, dx \right)^{\frac{1}{p}},
\end{equation}
with the convention that $R_0 = 1$ when unspecified: $X_p = X_p^1$.
This norm is inhomogeneous but scales similarly to $L^\infty$. This space with $p=2$ also appears in the literature for the Navier-Stokes equation in  2D~\cite{basson} and, with a  different weight, in 3D~\cite{BK1,BKT1}; a related weighted $L^2$ space appears in~\cite{FDLR1}.
For~\ref{eq:sqg}, if $\theta_0 = f(x/|x|)$ is self-similar with unbounded $f : S^{n-1} \to \R$, then $\th_0\notin L^p_\uloc$ for any $p<\I$, whereas $X_p$ admits such functions by choosing $f \in L^p(S^{n-1})$. Hence, the $X_p$ framework is natural for \ref{eq:sqg}, and we develop our local theory in the following norms built off of $X_2$: For $R_0,T > 0$, we define
\begin{equation*}
	\norm{f}_{A_T^{R_0}} := \esssup_{t \in (0,T)} \norm{f(\cdot,t)}_{X_2^{R_0}},
\end{equation*}
and
\begin{equation*}
	\norm{f}_{E_T^{R_0}}^2 := \sup_{R \geq R_0} \frac{1}{R^n}  \int_0^T \int_{B_R} |\Lambda^{1/2} f|^2 \, dx \, dt
\end{equation*}
with the convention that $R_0 = 1$ when unspecified.


\begin{theorem}[Global existence for unbounded data]\label{thrm.existence}
Let $p  > 2$ and $\theta_0 \in X_p$ be symmetric.

(Existence) There exists a global symmetric distributional solution $\th$ to~\eqref{eq:sqg} with $\theta, |\theta|^{p/2} \in A_T \cap E_T$ for all $T > 0$, satisfying the local energy inequalities
\begin{equation}
	\p_t |\theta|^2 + 2 \theta \Lambda \theta + \div (v |\theta|^2) \leq 0
\end{equation}
and
\begin{equation}
	\p_t |\theta|^p + 2 |\theta|^{p/2} \Lambda (|\theta|^{p/2}) + \div (v |\theta|^p) \leq 0
\end{equation}
in the sense of distributions with non-negative test functions, and attaining its initial data in the sense
\begin{equation}
	\| \th (\cdot,t)-\th_0\|_{L^p(K)}\to 0 \text{ as } t \to 0^+
\end{equation}
for every compact set $K \subset \R^n$.

(\emph{A priori} estimates) There exists $T_*=T_*(p,\|\th_0\|_{X_p}) > 0$ such that, for all $T\geq T_*$, any solution $\theta$ as above also satisfies the local energy estimates
\begin{equation}
	\norm{\theta}_{A_{T}^{T/T_*}}^2 + 	\norm{\theta}_{E_{T}^{T/T_*}}^2 \leq 2  \|\th_0\|_{X_2}^2
\end{equation}
and
\begin{equation}
	\norm{|\theta|^{p/2}}_{A_{T}^{T/T_*}}^2 + 	\norm{|\theta|^{p/2}}_{E_{T}^{T/T_*}}^2 \leq 2 \|\th_0\|_{X_p}^p.
\end{equation}

(Self-similarity) The solution $\theta$ as guaranteed to exist above may be chosen to satisfy the following property: If $\lambda > 1$ and $\th_0$ is $\lambda$-DSS, then $\theta$ is $\lambda$-DSS.
\end{theorem} 

Compared to \cite{ConstantinCordobaWu,AbidiHmidi,Lazar1,Lazar2}, the above solutions do not have any decay properties, may have unbounded data, and may grow as $|x| \to +\infty$.

When $\th_0\in X_p\setminus L^\I$, it is plausible that the solutions we construct become smooth after the initial time, since \eqref{eq:sqg} has a smoothing effect from $L^2$ initial data~\cite{CV}. It is also plausible that the above solutions are unique under the small data assumptions of Theorem~\ref{thm:moreregularsolutions}. We leave these and other extensions for future work. 

\subsection{Comparison with existing literature}

\subsubsection*{SQG theory} The weak solution theory of~\eqref{eq:sqg} was initiated by Resnick in~\cite{Resnick} (weak solutions with $L^2$ initial data) and further developed by Marchand in~\cite{MarchandCMP,MarchandPhysicaD} ($\dot H^{-1/2}$ and $L^p$, $p \geq 4/3$, solutions). The same papers considered also the inviscid SQG. Local energy solutions with initial data in $L^2_\uloc(\R^2)$ were developed by Marchand in~\cite[Theorem 1.4]{MarchandCMP} under the background assumptions $p \in (2,+\infty)$ and $\theta_0 \in L^p(\R^2)$. This theory was further explored by Lazar~\cite{Lazar1,Lazar2}, who considered initial data in the space $\Lambda^s (H^s_\uloc) \cap L^\infty$ and showed global existence of weak solutions when $s \in [1/4,1]$ and local existence when $s \in (0,1/4)$. These assumptions are made in order to define the Riesz transforms, but they exclude non-trivial DSS solutions, even for bounded data.\footnote{If $w_0 \in H^s_\uloc$ and $\th_0=\La^s w_0$ is self similar, then $\lam^s w_0(x/\lam)=w_0(x)$. If $w_0|_{\mathbb S^1}\neq 0$ on  a subset of $\mathbb S^1$ having positive 1D Lebesgue measure, then $w_0\notin L^2_\uloc$. But, $ H^s_\uloc\subset L^2_\uloc$, meaning Lazar's space excludes non-trivial self-similar data.}

To our knowledge, self-similar solutions of~\eqref{eq:sqg} were only considered by Lemari{\'e}-Rieusset and Marchand in~\cite{PGLR-SQG}, who constructed small self-similar solutions for initial data satisfying $\theta_0, \vec{R} \theta_0 \in L^\infty$. This is less regularity than we require in Theorem~\ref{thm:moreregularsolutions} but does not admit large data.\footnote{We wish to clarify the assumptions in~\cite[Theorem 4]{PGLR-SQG}. In order to unequivocally define $\vec{R} \theta_0$, Marchand and Lemari{\'e}-Rieusset assume that $P_{\leq j} \theta_0 \to 0$ in the sense of distributions as $j \to -\infty$, where $P_{\leq j}$ is a Littlewood-Paley projector. This way, $\vec{R}$ can be safely defined on the frequency block $P_j \theta_0$, and the blocks are summed afterward. This assumption does not admit non-trivial self-similar data. However, Theorem~4 in~\cite{PGLR-SQG} is completely valid within the symmetric class without this assumption. Interestingly, in Example 3 of~\cite{PGLR-SQG}, the authors gave a symmetric example of self-similar $\theta_0$ satisfying $\vec{R} \theta_0 \in L^\infty$.} Notice that the strong solution theories of Abidi and Hmidi~\cite{AbidiHmidi} in $\dot B^{0}_{\infty,1}$  and Miura~\cite{miuracriticalsobolevlarge} in $H^1$ do not admit non-zero self-similar solutions. Together with~\cite{PGLR-SQG}, these are the most general `strong solution' theories for~\eqref{eq:sqg} known to the authors. The condition $\norm{\theta_0}_{L^\infty} \ll 1$ was exploited by Constantin, C{\'o}rdoba and Wu in~\cite{ConstantinCordobaWu} to propagate regularity for all time. As we mentioned earlier, there are many proofs of global regularity, which fall into two categories: \emph{propagation of regularity} and \emph{smoothing}. For our purposes, we require the proofs of smoothing~\cite{CV,MaekawaMiura-drift}, since our initial data is `merely bounded.'

\subsubsection*{Navier-Stokes theory} Since Jia and {\v S}ver{\'a}k~\cite{jiasverakselfsim} demonstrated the existence of large self-similar Navier-Stokes solutions, there have been an abundance of works in this direction, including~\cite{tsaidiscretely,korobkovtsai,bradshawtsaiII,bradshawtsairot,lemarie2016,wolfchael2loc,bradshawtsaibesov,globalweakbesov}. The original proof in~\cite{jiasverakselfsim} is based on so-called \emph{local smoothing} near the initial time, whereas later proofs were based on (i) decompositions of the solution into an `approximate solution' (typically $e^{t\Delta} u_0$) and a finite-energy `correction', or (ii) local energy solutions. Analogues of these methods appear in Theorems~\ref{thm:moreregularsolutions} and~\ref{thrm.existence}, respectively. Our approach also utilizes a mollification procedure that preserves the self-similar scaling, see~\cite{wolfchael2loc,FDLR1}. Local energy solutions were originally introduced by Lemari{\'e}-Rieusset~\cite{lemarie2002} in the Navier-Stokes theory with further developments in~\cite{KikuchiSeregin,MMPEnergy,KwonTsai} and others. The issue with the definition of Riesz transforms for the pressure $p = R_i R_j (u_i u_j)$ without decay also features in this theory but is perhaps less dangerous, since only $\nabla p$ enters into the PDE. 


Our interest in non-uniqueness is partly motivated by the question, ``Could a Navier-Stokes solution lose uniqueness at a hypothetical singularity?" One may consider forward self-similar solutions as continuations of hypothetical \emph{backward self-similar solutions}, which were introduced by Leray in~\cite{leray}, see p. 225. While Tsai~\cite{tsai} demonstrated that backward self-similar singularities satisfying certain general assumptions do not exist, it is entirely plausible that forward discretely self-similar solutions (for example, with rotational correction~\cite{bradshawtsairot}) may arise in this way. 


\subsubsection*{Miscellaneous} Convex integration has played an important role in establishing rigorous non-uniqueness of Navier-Stokes~\cite{BuckmasterVicolAnnals}, SQG~\cite{IsettVicol,BuckmasterShkollerVlad,chengkwon2020non,isett2020direct} and QG~\cite{novackquasigeostrophic} solutions. In~\cite{elgindijoungsymmetries,elgindi2019singular,elgindijeong} by Elgindi and Jeong, $m$-fold rotational symmetry plays a crucial role in their well-posedness theory for the Euler equations and inviscid SQG. Global asymptotic stability of large self-similar solutions was recently established by Beekie and the first author in~\cite{albrittonbeekie2020long} for the critical Burgers equation $\p_t u + u \p_x u + \Lambda u = 0$ and multidimensional scalar conservation laws with critical dissipation.

\subsection{Strategy of the proof}

Restricting one's attention to symmetric solutions defines $\vec{R}^\perp \theta$ unequivocally. Afterward, the strategy is as follows:

\subsubsection*{Bounded solutions} One wishes to leverage the conserved (or almost conserved) quantities of the equation. We focus on the critical $L^\infty$ norm, which is monotone due to the maximum principle. In this respect, our analysis is simpler than that for the Navier-Stokes equations, based on supercritical quantities. Next, one bootstraps from $L^\infty_x \to C^\alpha_x$ using the De Giorgi-type estimates of Caffarelli-Vasseur~\cite{CV} and, subsequently, from $C^\alpha_x \to C^{1,\alpha}_x$. These estimates are linear in nature.\footnote{There is a related but different $C^\alpha_x$ smoothing effect at work in~\cite{jiasverakselfsim}.} Less obvious is the control of the solution in time, since $\vec{R}^\perp \theta$ is only controlled in $\BMO$. As mentioned before, some care is required in the linear theory to ensure that every bounded smooth solution satisfies the maximum principle, that is, the $L^\infty$ norm may not `jump' at the initial time. This is developed in Lemma~\ref{lem:Linftytheory}. Once the \emph{a priori} estimates are in place, the self-similar solutions are constructed by the Leray-Schauder fixed point theorem at the level of an approximate equation with small parameter $\delta \to 0^+$. 

\subsubsection*{Regular solutions} Often, it is more convenient to work with energy-type quantities. For this, one wishes to introduce some decay into the problem. Let $\theta^{\rm lin} = e^{-t\Lambda} \theta_0$. We perturb off of the solution of the half-heat equation: $\theta = \theta^{\rm lin} + \psi$. The equation satisfied by the remainder $\psi$ is
\begin{equation}
	\label{eq:psiequation}
\begin{aligned}
	\p_t \psi + \Lambda \psi + \vec{R}^\perp \theta \cdot \nabla \psi + \vec{R}^\perp \psi \cdot \nabla \theta^\lin &= - \vec{R}^\perp \theta^\lin \cdot \nabla \theta^\lin
	\end{aligned}
\end{equation}
with $\psi(\cdot,0) = 0$. Under our assumptions, the (critical) forcing term on the RHS belongs to $L^\infty_{t,\loc} L^p_x$ (with suitable time weights) for all $p >2 $. Hence, we expect that the solution $\theta$ also belongs to $L^\infty_{t,\loc} L^p_x$ (with suitable time weights). The main difficulty here is that the potential $\nabla \theta^\lin$ in the term $\vec{R}^\perp \psi \cdot \nabla \theta^\lin$, which we would like to regard as `lower order', is actually \emph{large in critical spaces}.\footnote{As $\theta^\lin$ is self-similar, it only belongs to `singularly critical' or `ultracritical' spaces, which do not have a hidden smallness. For example, in this Navier-Stokes theory, consider the difference between $L^5_{t,x}$ and $L^{5,\infty}_t L^5_x$.} This difficulty makes the $L^p$ energy estimates unclear at the level of $\psi$. There is a known technique in the Navier-Stokes literature for circumventing this difficulty: the Calder{\'o}n-type splitting~\cite{calderon90}. We split the critical initial data $\theta_0 = \theta^{\rm sub}_0 + \theta^{\rm sup}_0$ into subcritical (H{\"o}lder continuous) and supercritical parts. Subsequently, one splits the solution as $\theta = e^{-t \Lambda} \theta^{\rm sub}_0 + \theta^{\rm sup}$. The coefficients in the equation for $\theta^{\rm sup}$ are subcritical, and it is possible to close an $L^p$ energy estimate. With such estimates in hand, one proves weak-strong uniqueness for small data in nearly the same way as in the Navier-Stokes theory. 

\subsubsection*{Unbounded solutions} While self-similar solutions have infinite energy, it is possible to control them via localized energy estimates in the function space $X_p$ mentioned above. 
In the estimates, there are boundary terms (from both the dissipation and drift) due to the flux of energy between annuli, but this is not a problem, since one controls all annuli at once. To control $\vec{R}^\perp \theta$, we use the simple observation that, in the presence of symmetries, $\vec{R}^\perp \: X_p \to X_p$ for all $p \in (1,+\infty)$. The restriction $p>2$ in Theorem~\ref{thrm.existence} comes from the term
\begin{equation}
	\iint |\theta|^2 v \cdot \nabla \varphi \, dx \, dt
\end{equation}
in the $L^2$ energy estimates.
Since $\dot H^{1/2} \subset L^4$, one can afford to (locally) place $v \in L^\infty_t L^p_x$, $p > 2$, and still `absorb' the bulk of $|\theta|^2$ into the LHS of the energy estimate. The same considerations give $L^{5/2}_{t,x}$ as the borderline quantity in the $\varepsilon$-regularity theory for~\eqref{eq:NS}~\cite{GustafsonKangTsai,wang2019regularity}.


\subsection{Self-similar Navier-Stokes solutions in two dimensions}
Consider the stream function
\begin{equation}
	\label{eq:initialstreamfunctions}
	\tilde{\psi}_0 = \psi_0 + \frac{\alpha}{2\pi} \log |x|,
\end{equation}
where $\psi_0$ is $0$-homogeneous and $\alpha \in \R$. Does there exist a self-similar solution of~\eqref{eq:NS2d} with initial vorticity $\omega_0 = \Delta \tilde{\psi}_0$? When $\psi_0 = 0$, the corresponding self-similar solutions are the \emph{Lamb-Oseen vortices}, whose initial data are Dirac masses. The Lamb-Oseen vortices were shown by Gallay and Wayne~\cite{GallayWayne} to be the unique solutions, within a natural class, to~\eqref{eq:NS2d} with Dirac mass initial data and to describe the long-time behavior of~\eqref{eq:NS2d}. Non-uniqueness of solutions to~\eqref{eq:NS2d} is anticipated for general initial stream functions of the form~\eqref{eq:initialstreamfunctions}, yet the proof of existence is unknown even when $\alpha = 0$.

One difficulty is that there is no quantity controlled `for free' from which to begin the analysis. At the level of the velocity field, the typical quantity is an $L^2$-based norm, such as $L^2_\uloc$, whereas non-trivial $-1$-homogeneous initial data $u_0$ does not belong to $L^2_\loc$ in dimension two. At the level of the vorticity, the $L^1$, $L^\infty$, and Lorentz space $L^{p,q}$ norms, $p \in (1,+\infty)$, $q \in [1,+\infty]$, would be controlled with estimates independent of the drift, except that our desired $\omega$ `only' belongs to $L^{1,\infty}$, $\nabla^\perp L^{2,\infty}(\R^2;\R^2)$, etc. That is, the obviously controlled quantities are either critical and `just missing' the initial data or subcritical, whereas known methods require control on supercritical or good critical quantities. We intend to address these issues in future work.

The equations~\eqref{eq:NS2d} and~\eqref{eq:sqg} belong to the family of \emph{generalized dissipative SQG equations}:
\begin{equation}
	\label{eq:msqg}
	\tag{gSQG}
	\p_t \omega + (\nabla^\perp \psi) \cdot \nabla \omega + \Lambda^\zeta \omega = 0, \quad \Lambda^s \psi = \omega,
\end{equation}
where $s, \zeta \in (0,2]$. Choosing $(s,\zeta) = (2,2)$ and switching $\omega \to - \omega$ recovers~\eqref{eq:NS2d}, while $(s,\zeta) = (1,1)$ recovers~\eqref{eq:sqg}. These equations were considered in~\cite{Resnick,ConstantinCordobaGancedoWu}. The Hamiltonian $H = \int \psi \omega \, dx = 0$, which corresponds to the kinetic energy in~\eqref{eq:NS2d} and the $\dot H^{-1/2}$ norm in~\eqref{eq:sqg}, is conserved for the inviscid dynamics. It would be interesting to know whether one can extend the present work by exploiting this quantity. We expect that the methods in the present work can be extended to~\eqref{eq:msqg} for some range of exponents. Furthermore, it may be interesting to understand whether the viscous point vortex solutions to~\eqref{eq:msqg}, analogous to the Lamb-Oseen vortices in~\eqref{eq:NS2d}, are unique. 

\subsection{On numerics}

It is reasonable to expect that one may detect symmetry-breaking numerically as in~\cite{guillodsverak}, which is based on numerical continuation for the steady-state problem~\eqref{eq:NSinselfsim} in similarity variables. In our setting, the numerics may be complicated by (i) non-locality appearing in the non-linearity and in the diffusion, and (ii) the boundary condition $\theta_0^{\rm ss}(x)$ as $|x| \to +\infty$ rather than decay to zero. It may also be possible to detect non-uniqueness by solving the time-dependent problem in similarity variables $y=x/t$, $s = \log t$, since a unique self-similar solution is stable (Theorem~\ref{thm:condasstab}) and presumably exchanges stability in the symmetry breaking. In this vein, we remark that algebraic decay in physical variables becomes exponential decay in similarity variables.


\section{Linear theory}

For notational convenience, we allow all constants to depend implicitly on the dimension $n \geq 2$.

\subsection{Non-local operators}

Let $\vec{R} = \vec{\nabla} (-\Delta)^{-1/2}$ be the vector of Riesz transforms. These operators are well defined as Fourier multipliers mapping $L^2 \to L^2$ and, by the classical Calder{\'o}n-Zygmund theory, may be extended as operators $L^p \to L^p$ ($1 < p < +\infty$) and $L^1 \to L^{1,\infty}$. In our setting, we require the following extension $L^\infty \to \BMO$ (see Stein \cite{BigStein}, p. 155--157; Duoandikoetxea~\cite{Duoand}, p. 118--119) that defines the Riesz transforms \emph{up to constants}:
\begin{equation}
	\label{eq:Riesztransformdef}
	\vec{R}f (x) = [\vec{R} (\mathbf{1}_{B_1} f)](x) + \pv \int_{\R^n} [K(x,y) - K(0,y)] (\mathbf{1}_{\R^n \setminus B_1} f)(y) \,dy,
\end{equation}
where
\begin{equation}
	K(y)=c_n \frac{y}{|y|^{n+1}}
\end{equation}
is the associated kernel. Definition~\eqref{eq:Riesztransformdef} is valid pointwise a.e. for any function $f \in L^1_\loc(\R^n)$ whose tail converges absolutely against $1/\la y \ra^{n+1}$. In particular, $\vec{R}$ is valid on the spaces $X_p$, $p \geq 1$, which we define below:


We use the weighted spaces $X_p^{R_0}$ and $X_{p,\osc}^{R_0}$ defined by the norms
\begin{equation}
	\norm{f}_{X_p^{R_0}} :=\sup_{R\geq R_0} \left( \barint_{B_R} |f|^p \, dx \right)^{\frac{1}{p}},
\end{equation}
and 
\begin{equation}
	\norm{f}_{X_{p,\osc}^{R_0}} :=\sup_{R\geq R_0} \left( \barint_{B_R}\left\lvert f(x)-\barint_{B_R} f(y) \, dy\right\rvert^p \, dx \right)^{\frac{1}{p}},
\end{equation}
where $R_0>0$ and $1\leq p<\I$. If $R_0=1$ we write $X_p =X_p^{R_0}$.
We set $X_\I=L^\I$ for notational convenience.
We have the trivial inclusions $X_{p_1} \into X_{p_2}$, $p_1 \geq p_2$, $L^\infty(\R^n) \into X_p$, and
$\BMO(\R^n) \into {X_{p,\osc}}$.

Let $f \in X_1$. For $j \in \N$, we define $P_j f = \mathbf{1}_{B(2^j)} f - \mathbf{1}_{B(2^{j-1})} f$ and $P_{\leq j} f = \mathbf{1}_{B(2^j)} f$. Then
\begin{equation}
	f = P_{\leq 1} f + \sum_{j=2}^\infty P_j f \text{ in } L^1_\loc(\R^n),
\end{equation}
and\footnote{This is a routine computation in the theory of Besov spaces of negative order. The non-obvious direction is
\begin{equation}
	\norm{P_j f}_{L^p(\R^n)} \leq \sum_{1 \leq k \leq j} 2^{nk/p} 2^{-nk/p} \norm{P_k f}_{L^p(\R^n)} \les_p  \sum_{k \leq j} 2^{nk/p} \times \sup_{\ell \geq 1} 2^{n\ell/p} \norm{P_{\ell} f}_{L^p(\R^n)} \les_p 2^{nj/p} \norm{f}_{X_p}. 
\end{equation}
The same computation underlies the heat characterization of negative order Besov spaces.
}
\begin{equation}
	\norm{f}_{X_p} \sim_p \norm{P_{\leq 1} f}_{L^p(\R^n)} + \sup_{j \geq 2} \, 2^{-nj/p} \norm{P_j f}_{L^p(\R^n)}.
\end{equation}

\begin{lemma}[Riesz transforms on $X_p$]
\label{lem:riesz}
For all $p \in (1,+\infty)$, we have $\vec{R} \: X_p \to X_{p,\osc}$ with
\begin{equation}
	\label{eq:xposcbound}
	\| \vec R f \|_{X_{p,\osc}} \lesssim_p \|f\|_{X_p}.
\end{equation}
Assume further that $n=2$ and $f$ is symmetric. Then $\vec{R} f \in X_p$ (defined to be the unique symmetric representative) and $\norm{\vec{R}f}_{X_p} = \norm{\vec{R}f}_{X_{p,\osc}}$. In particular,
\begin{equation}
	\label{eq:xposcbound2}
	\| \vec R f \|_{X_{p}} \lesssim_p \|f\|_{X_p}.
\end{equation}
\end{lemma}

The symmetry ensures that $\int_{B_R} \vec{R} f \, dx = 0$ for all $R > 0$, as justified in~\eqref{eq:quickrotationcomputation}, so the estimate~\eqref{eq:xposcbound2} for symmetric $f$ follows immediately from~\eqref{eq:xposcbound} and the definitions of $X_p$ and $X_{p,\osc}$.

\begin{proof}
Let $f \in X_p$. We use the modified Riesz transform definition in~\eqref{eq:Riesztransformdef}. Our goal is to estimate $2^{-j_0n/p} \norm{P_{\leq j_0} \vec{R} f}_{L^p(\R^n)}$ uniformly in $j_0 \in \N$. Rewriting the singular integral as a sum over dyadic shells gives
\begin{equation}
	\label{eq:rieszequality}
	\vec{R} f = \vec{R} P_{\leq {j_0+1}} f + \sum_{j={j_0+2}}^{+\infty} \vec{R} P_{j} f \text{ in } L^1_\loc(\R^n).
\end{equation}
We apply $P_{\leq j_0}$ to~\eqref{eq:rieszequality} and estimate each term. First,
\begin{equation}
	\label{eq:first}
	\norm{P_{\leq j_0} \vec{R} P_{\leq j_0+1} f}_{L^p(\R^n)} \les_p \norm{P_{\leq j_0+1} f}_{L^p(\R^n)} \les_p 2^{nj_0/p} \norm{f}_{X_p}.
\end{equation}
Second, when $j \geq j_0+2$, we have
\begin{equation}
	\label{eq:linftybdonp}
	|(P_{\leq j_0} \vec{R} P_{j} f)(x)| \les 2^{j_0-j} 2^{-nj} \norm{P_{j} f}_{L^1(\R^n)} \overset{\eqref{eq:lebesgueintfact}}{\les} \underbrace{2^{j_0-j} 2^{-nj} 2^{nj(1-1/p)}}_{= 2^{j_0-j} 2^{-nj/p}} \norm{P_{j} f}_{L^p(\R^n)},
\end{equation}
since, when $|x| \leq 2^{j_0}$ and $|y| \geq 2^{j-1}$, we have $|x-y| \ges 2^j$, and the kernel $\tilde{K}(x,y) = K(x,y) \mathbf{1}_{B_1} + [K(x,y) - K(0,y)] \mathbf{1}_{\R^n \setminus B_1}$ of the modified Riesz transform $\vec{R}$ in~\eqref{eq:Riesztransformdef} satisfies $|\tilde{K}(x,y)| \les |x| |x-y|^{-n} |y|^{-1}$ when $y \geq 2$.\footnote{Recall that
\begin{equation}
	\label{eq:lebesgueintfact}
	\norm{g}_{L^\ell(\Omega)} \leq |\Omega|^{1/\ell-1/s} \norm{g}_{L^s(\Omega)}
\end{equation}
when $1 \leq \ell \leq s \leq +\infty$, $\Omega \subset \R^n$ is measurable with $|\Omega| < +\infty$, and $g \in L^1(\Omega)$.}
In particular, due to $\supp P_{\leq j_0} f \subset B(2^{j_0})$ and~\eqref{eq:lebesgueintfact}, we have
\begin{equation}
	\label{eq:second}
\begin{aligned}
	\norm{P_{\leq j_0} \vec{R} P_j f}_{L^p(\R^n)} \les 2^{nj_0/p} 2^{j_0-j} 2^{-nj/p} \norm{P_{j} f}_{L^p(\R^n)} \les 2^{nj_0/p} 2^{j_0-j} \norm{f}_{X_p}.
	\end{aligned}
\end{equation}
Finally, combining~\eqref{eq:rieszequality},~\eqref{eq:first}, and~\eqref{eq:second}, we have
\begin{equation}
	2^{-nj_0/p} \norm{P_{\leq j_0} \vec{R} f}_{L^p(\R^n)} \les_p \left( 1 + \sum_{j=j_0+2}^{+\infty} 2^{j_0-j} \right) \norm{f}_{X_p} \les_p \norm{f}_{X_p}.
\end{equation}
\end{proof}

\begin{remark}[Riesz transforms and convergence]
\label{rmk:riesztransformsandconvergence}
We require the following adaptation of Lemma~\ref{lem:riesz}. Let $T>0$.
Assume that $f^{(k)} \in L^\infty_t (X_p)_x(\R^n \times (0,T))$ is symmetric for all $k\in \N$ , $f^{(k)} \wstar f$ in $L^\infty_t (X_p)_x(\R^n \times (0,T))$, and $f^{(k)} \to f$ in $L^2_\loc(\R^n \times [0,T])$  as $k \to +\infty$. Pass to a subsequence (still labeled by $k$) satisfying $\vec{R}^\perp f^{(k)} \wstar g$ in $L^\infty_t (X_p)_x(\R^n \times (0,T))$ as $k \to +\infty$. Then $g = \vec{R}^\perp f$.
\end{remark}

\begin{remark}[On weak fractional derivatives]
	 When $\varphi \in L^\infty \cap C^2_\loc$, the fractional Laplacian $\Lambda^{s} \varphi = (-\Delta)^{\frac{s}{2}} \varphi$, where $s \in (0,2)$, is defined by
	\begin{equation}
		\Lambda^s \varphi(x) = c_{s,n} \pv \int_{\R^n} \frac{\varphi(x)-\varphi(y)}{|x-y|^{n+s}} \, dy.
	\end{equation}
	The above exponent is dimensionally correct. While $\Lambda^s \theta$ is not well behaved at the level of tempered distributions, we may provide a weak definition when $\theta$ belongs to the restricted class $\theta \in X_1$. Then $\Lambda^s \theta$ is a tempered distribution defined by duality:
	\begin{equation}
		\la \Lambda^s \theta, \varphi \ra = \int_{\R^n} \theta(x) \cdot \Lambda^s \varphi(x) \, dx,
	\end{equation}
	where $\varphi$ belongs to the Schwartz class. The above integral makes sense because
	\begin{equation}
	|\Lambda^s \varphi(x)| \les_{\varphi} 1/\la x \ra^{n+s},
	\end{equation}
	and
	\begin{equation}
	\int_{\R^n} |\theta(x)| \la x \ra^{-(n+s)} \, dx \les \sum_{k \geq 0} 2^{-k(n+s)} \int_{B_{2^{k+1} \setminus 2^k}} |\theta(x)| \, dx \les \sum_{k \geq 0} 2^{-ks} \norm{\theta}_{X_1} < +\infty.
	\end{equation}
	For time-dependent functions on $\R^n \times (0,T)$, we have an analogous definition when
	\begin{equation}
	\sup_{R \geq 1} \int_0^T \barint_{B_R} |\theta(x)| \, dx < +\infty.
	\end{equation}
\end{remark}

\subsection{Non-local drift-diffusion equation}
\label{sec:linftytheory}

In this section, we sketch a solution theory for the linear PDE
\begin{equation}
	\label{eq:basiclinearpde}
	\p_t \theta + v \cdot \nabla \theta + \Lambda \theta = 0,
\end{equation}
where $\div v = 0$. We are particularly interested in bounded solutions, whose theory is summarized in Lemma~\ref{lem:Linftytheory}. For unbounded solutions, we develop local energy estimates in Section~\ref{sec:localenergy}.

To begin, we sketch the $L^2$ theory, where existence, uniqueness, and the H{\"o}lder smoothing estimates of Caffarelli and Vasseur~\cite{CV} are readily justified. Afterward, we pass to the $L^\infty$ setting.

Let $Q_T = \R^n \times (0,T)$.

\begin{lemma}[$L^2$ theory]
	\label{lem:l2theory}
Let $v \in L^\infty_t C^{1/2}_x(Q_1)$ and $\div v = 0$. Let $\theta_0 \in L^1 \cap L^\infty(\R^n)$. Then there exists a unique solution $\theta \in L^\infty_t L^2_x \cap L^2_t H^{1/2}_x(Q_1)$ to~\eqref{eq:basiclinearpde} satisfying $\theta(\cdot,t) \to \theta_0$ in $L^2(\R^n)$ as $t \to 0^+$. Additionally, $\theta$ belongs to $C([0,1];L^2(\R^n))$, and for all $p \in [1,+\infty]$,
\begin{equation}
	\norm{\theta}_{L^\infty_t L^p_x(Q_T)} \leq \norm{\theta_0}_{L^p(\R^n)}.
\end{equation} 
\end{lemma}

\begin{proof}[Proof sketch]
Consider the drift-diffusion equation
\begin{equation}
	\label{eq:viscousapproximation}
	\p_t \theta^\varepsilon + v \cdot \nabla \theta^\varepsilon + \Lambda \theta^\varepsilon = \varepsilon \Delta \theta^\varepsilon.
\end{equation}
Under the given assumptions, there exists a unique solution $\theta^\varepsilon \in C([0,1];L^2(\R^n)) \cap L^2_t H^1_x(Q_1)$ with $\theta^\varepsilon(\cdot,0) = \theta_0$. 
The unique solution also satisfies $\norm{\theta^\varepsilon}_{L^p(Q_1)} \leq \norm{\theta_0}_{L^p(\R^n)}$ for all $p \in [1,+\infty]$. For $p = +\infty$, this is a consequence of the maximum principle. The case $p=1$ may be seen as a consequence of the maximum principle for the dual problem. Upon sending $\varepsilon \to 0^+$, we obtain a solution $\theta \in L^\infty_t L^2_x \cap L^2_t H^{1/2}_x(Q_1)$ with $\theta(\cdot,t) \to \theta_0$ in $L^2(\R^n)$ as $t \to 0^+$ and satisfying the desired estimates. This type of argument is well known from the Navier-Stokes theory. We now demonstrate the energy equality and uniqueness. Consider a solution as above. Since $v \in L^\infty_t C^{1/2}_x(Q_1)$, we have $v \cdot \nabla \theta \in L^2_t H^{-1/2}_x(Q_1)$. Clearly, $\Lambda \theta$ belongs to $L^2_t H^{-1/2}_x(Q_1)$. Hence,  $\p_t \theta \in L^2_t H^{-1/2}_x(Q_1)$. These assumptions give that $\theta(\cdot,t) \in C_{\rm wk}([0,1];L^2(\R^n))$, $\norm{\theta(\cdot,t)}_{L^2(\R^n)}^2 \in C([0,1])$, and $\theta \in C([0,1];L^2(\R^n))$. Now we may justify integrating~\eqref{eq:basiclinearpde} against $\theta \mathbf{1}_{(t_0,t_1)}$, where $0 \leq t_0 < t_1 \leq 1$:
\begin{equation}
	\norm{\theta(\cdot,t_1)}_{L^2(\R^n)}^2 - \norm{\theta(\cdot,t_0)}_{L^2(\R^n)}^2  + 2 \int_{t_0}^{t_1} \int_{\R^n} |\Lambda^{1/2} \theta|^2 \, dx \, dt =  0,
\end{equation}
since $\la v \cdot \nabla \theta, \theta \ra = 0$. If $\theta_0 = 0$, the above equality gives $\theta \equiv 0$ and uniqueness.
\end{proof}

\begin{remark}[On energy equality]
\label{rmk:onenergyequality}
The assumption $v \in L^\infty_t C^{1/2}_x(Q_1)$ is qualitative. The key point is to make sense of $\iint (v \cdot \nabla \theta) \theta \, dx \, dt$. Heuristically, since $\theta$ has $1/2$-derivative in $L^2$, we wish to distribute $\nabla$ among the two copies of $\theta$. This begets a commutator $[\Lambda^{1/2}, v] : L^2(\R^n) \to L^2(\R^n)$ from moving $1/2$-derivative through $v$. If instead one wishes to distribute $1/3$-derivative on each term, one arrives at the criterion in Constantin-E-Titi~\cite{ConstantinETiti}.

It is possible to move away from the H{\"o}lder-$1/2$ condition in the Leray-Hopf setting of~\eqref{eq:sqg} with the additional assumption $\theta \in L^\infty_{t,x}(Q_1)$. Let $v = \vec{R}^\perp \theta \in L^\infty_t \BMO_x \cap L^\infty_t L^2_x \cap L^2_t H^{1/2}_x(Q_1)$. We expand (and substitute $\nabla = \vec{R} \Lambda^{1/2}$):
\begin{equation}
\begin{aligned}
	&\iint (v \cdot \nabla \theta) \theta \, dx \, dt = \iint \left[ \Lambda^{1/2}(\theta v) - (\Lambda^{1/2} \theta) v - \theta (\Lambda^{1/2} v) \right] \cdot (\vec{R} \Lambda^{1/2} \theta )\, dx \, dt \\
	&\quad + \iint \theta \Lambda^{1/2} v \cdot (\vec{R} \Lambda^{1/2} \theta) \, dx \, dt + \iint v \Lambda^{1/2} \theta \cdot (\vec{R} \Lambda^{1/2} \theta) \, dx \, dt.
	\end{aligned}
\end{equation}
Each term on the RHS makes sense. First, we have an endpoint fractional Leibniz estimate (see (1.7) in~\cite{DongLiFractional}):
\begin{equation}
	\norm{ \Lambda^{1/2}(\theta v) - (\Lambda^{1/2} \theta) v - \theta (\Lambda^{1/2} v) }_{L^2_{t,x}(Q_1)} \les \norm{\Lambda^{1/2} \theta}_{L^2_{t,x}(Q_1)} \norm{v}_{L^\infty_t \BMO_x(Q_1)}.
\end{equation}
The second term is estimated routinely:
\begin{equation}
	\iint |\theta \Lambda^{1/2} v \cdot (\vec{R} \Lambda^{1/2} \theta) | \, dx \,dt \les \norm{\theta}_{L^\infty_{t,x}(Q_1)} \norm{\Lambda^{1/2} v}_{L^2_{t,x}(Q_1)} \norm{\Lambda^{1/2} \theta}_{L^2_{t,x}(Q_1)}.
\end{equation}
The third term is estimated by $\mathcal{H}^1$/$\BMO$ duality:
\begin{equation}
	\left| \iint v \Lambda^{1/2} \theta \cdot (\vec{R} \Lambda^{1/2} \theta) \, dx \, dt \right| \les \norm{v}_{L^\infty_t \BMO_x} \norm{\Lambda^{1/2} \theta \vec{R} \Lambda^{1/2}\theta }_{L^1_t \mathcal{H}^1_x(Q_1)},
\end{equation}
since $f \vec{R} f \in \mathcal{H}^1(\R^n)$ when $f \in L^2(\R^n)$~\cite{coifmanrochbergweiss}.

\end{remark}

We now apply the main theorem of Caffarelli and Vasseur in~\cite{CV}. 

\begin{lemma}[H{\"o}lder continuity]
\label{lem:Holdercontinuity}
Let $\div v = 0$ and
\begin{equation}
	\label{eq:driftB}
	\norm{v}_{L^\infty_t \BMO_x(Q_1)} \leq B.
\end{equation}
Then there exists $\alpha_0 = \alpha_0(B) \in (0,1)$ such that, for all $t \in (0,1]$, the unique solution $\theta$ from Lemma~\ref{lem:l2theory} satisfies
\begin{equation}
	\label{eq:initialholder}
	t^{\alpha_0} \norm{\theta(\cdot,t)}_{C^{\alpha_0}(\R^n)} \les_B \norm{\theta_0}_{L^\infty(\R^n)}.
\end{equation}
\end{lemma}

Lemma~\ref{lem:Holdercontinuity} can also be recovered from the fundamental solution bounds of Maekawa and Miura in~\cite[Theorem 1.2]{MaekawaMiura-drift}.

\begin{lemma}[$C^{1,\alpha_0}_x$ regularity]
\label{lem:higherreg}
Let~\eqref{eq:driftB} be satisfied and, for all $t \in (0,1]$,
\begin{equation}
	\label{eq:driftN}
	t^{\alpha_0} [v(\cdot,t)]_{C^{\alpha_0}(\R^n)} \leq N,
\end{equation}
where $\alpha_0$ is the exponent in Lemma~\ref{lem:Holdercontinuity}.
Then the unique solution $\theta$ in Lemma~\ref{lem:l2theory} satisfies
\begin{equation}
	\label{eq:higherholder}
	t^{1+\alpha_0} [\theta(\cdot,t)]_{C^{1,\alpha_0}_x(\R^n)} \les_{B,N} \norm{\theta_0}_{L^\infty(\R^n)}.
\end{equation}
\end{lemma}
\begin{proof}
Recall Silvestre's estimate in~\cite[Theorem 1.1]{SilvestreHigher}:
\begin{equation}
	\label{eq:silvestreestimate}
	\norm{\theta}_{L^\infty_t C^{1,\alpha}_x(B_{1/2} \times (-1/2,1))} \leq C \left[ \norm{\theta}_{L^\infty_{t,x}(\R^n \times (-1,0))} + \norm{f}_{L^\infty_t C^\alpha_x(\R^n \times (-1,0))} \right],
\end{equation}
for solutions of~\eqref{eq:basiclinearpde} on $B_1 \times (-1,0)$ with RHS $f$.\footnote{Since the above energy class solutions are unique, one may use the viscous approximation~\eqref{eq:viscousapproximation} to justify the  application of Silvestre's estimates in~\cite{SilvestreHigher}. This is discussed in Section~3.2 of~\cite{SilvestreHigher}.} In this proof, $f=0$, but we mention it to use in Lemma~\ref{lem:higherspatialreg}. Notably, Silvestre's constant $C$ depends on $\alpha \in (0,1)$ and the \emph{inhomogeneous} norm $\norm{v}_{L^\infty_t C^{\alpha}_x(B_1 \times (-1,0))}$, whereas in~\eqref{eq:driftN} we control only the H{\"o}lder \emph{seminorm}.\footnote{Recall also that, in~\cite{SilvestreHigher}, $v$ may have non-zero divergence.} We will demonstrate how to get around this. Let $\theta_0 \in C^\infty_0(\R^n)$ and $v$ be smooth, compactly supported, and divergence free.  We will apply~\eqref{eq:silvestreestimate} to justify
\begin{equation}
	\label{eq:newestimatetojustify}
	[\theta(\cdot,1)]_{C^{1,\alpha_0}_x(\R^n)} \leq C \norm{\theta}_{L^\infty_{t,x}(\R^n \times (1/2,1))},
\end{equation}
where $C$ depends on $\alpha_0$ and $\sup_{t \in (1/2,1)} [v(\cdot,t)]_{C^{\alpha_0}(\R^n)} \les N$. Then~\eqref{eq:higherholder} will follow from the maximum principle, scaling invariance, and a weak-$\ast$ approximation argument. Notice the translation invariance of the norms in~\eqref{eq:newestimatetojustify}. Let $x_0 \in \R^n$ and $y(t)$ be the solution of the ODE $\dot y(t) = v(y(t),t)$ with $y(1) = x_0$. To estimate $[\theta(\cdot,1)]_{C^{1,\alpha_0}(B(x_0,1/4))}$, we consider $\tilde{\theta}(x,t) = \theta(x+y(t),t)$ and $\tilde{v}(x,t) = v(x+y(t),t) - \dot y(t)$, which solve the PDE
\begin{equation}
	\p_t \tilde{\theta} + \tilde{v} \cdot \nabla \tilde{\theta} + \Lambda \tilde{\theta} = 0.
\end{equation}
Notice that $\tilde{v}(0,t) = 0$ and $[\tilde{v}(\cdot,t)]_{C^{\alpha_0}(\R^n)} = [v(\cdot,t)]_{C^{\alpha_0}(\R^n)}$. Therefore, $\norm{\tilde{v}}_{L^\infty_{t,x}(B_{1/2} \times (1/2,1))}$ is controlled by $\sup_{t \in (1/2,1)} [\tilde{v}]_{C^{\alpha_0}(B_{1/2} \times (1/2,1))}$. Hence, we may apply a translated and rescaled version of Silvestre's estimate~\eqref{eq:silvestreestimate} to $\tilde{v}$. Since $x_0$ was arbitrary, we obtain~\eqref{eq:newestimatetojustify}. The proof is complete.
\end{proof}

In particular, the equation~\eqref{eq:basiclinearpde} is satisfied pointwise a.e. in $Q_1$.

The main result of this section is

\begin{lemma}[$L^\infty$ theory]
\label{lem:Linftytheory}
Let the divergence-free drift $v$ satisfy the $\BMO$ estimate~\eqref{eq:driftB} and the H{\"o}lder continuity estimate~\eqref{eq:driftN}. We do not ask that $v \in L^\infty_t C^{1/2}_x(Q_1)$. Assume that $p > n$ and
\begin{equation}
	\label{eq:driftV}
	\norm{v}_{L^\infty_t (X_p)_x(Q_1)} \leq V.
\end{equation}

(Existence) Let $\theta_0 \in L^\infty(\R^n)$. Then there exists a solution $\theta \in L^\infty_{t,x}(Q_1)$ of~\eqref{eq:basiclinearpde} that satisfies the maximum principle
\begin{equation}
	\norm{\theta}_{L^\infty_{t,x}(Q_1)} \leq \norm{\theta_0}_{L^\infty(\R^n)}
\end{equation}
and the \emph{a priori} H{\"o}lder estimates~\eqref{eq:initialholder} and~\eqref{eq:higherholder}. This solution attains its initial data in the sense
\begin{equation}\label{eq:attaininitialdatabounded}
	\theta(\cdot,t) \wstar \theta_0 \text{ in } L^\infty(\R^n)
\end{equation}
as $t \to 0^+$.

 (Uniqueness) Let $\theta \in L^\infty_{t,x}(Q_1)$ be a solution of~\eqref{eq:basiclinearpde} with $\theta(\cdot,t) \wstar 0$ as $t \to 0^+$ and satisfying that, for all compact $K = K_1 \times K_2 \subset Q_1$, there exists $\beta \in (0,1)$ such that $\theta \in L^\infty_t C^{1,\beta}_x(K)$. Then $\theta \equiv 0$. From now on, `solution' refers to the unique solution in this class.

 (Continuity by compactness and uniqueness) Let $\theta^{(k)}$, $k \in \N$, be solutions to~\eqref{eq:basiclinearpde} with divergence-free drifts $v^{(k)}$ satisfying
 \begin{equation}
	v^{(k)} \wstar v \text{ in the sense of distributions}
 \end{equation}
 and initial data
  \begin{equation}
	\theta^{(k)}_0 \wstar \theta_0 \text{ in } L^\infty(\R^n).
 \end{equation}
 Assume that each $v^{(k)}$ satisfies the inequalities~\eqref{eq:driftB},~\eqref{eq:driftN}, and~\eqref{eq:driftV}. (By lower semi-continuity of the relevant norms, $v$ satisfies the same inequalities.)
Then $\theta^{(k)}$ converges to the solution $\theta$ of~\eqref{eq:basiclinearpde} with drift $v$ and initial data $\theta_0$, obtained as in \eqref{eq:attaininitialdatabounded}, in the following senses (among others):
 \begin{equation}
 	\label{eq:linftyconv}
	\theta^{(k)} \wstar \theta \text{ in } L^\infty_{t,x}(Q_1),
 \end{equation}
 \begin{equation}
 	\label{eq:betterconv}
	\theta^{(k)} \to \theta \text{ in } L^\infty_t C^{1,\beta}_x(K)
 \end{equation}
 for all $\beta \in (0,\alpha_0)$ and on all compact sets $K = K_1 \times K_2 \subset \R^n \times (0,1]$.
\end{lemma}

It is possible to prove a more quantitative stability theorem than the one above, but it is not necessary here.

\begin{proof}
(Existence) This follows from an approximation argument with initial data $\theta_0^{(k)} \in L^1 \cap L^\infty(\R^n)$ and drifts $v^{(k)} \in L^\infty_t C^{1/2}_x(Q_1)$, where $k \in \N$.

(Uniqueness) We prove uniqueness by a duality argument, which the first author employed in a similar context in~\cite{albrittonbeekie2020long}. The crux of the matter is that the initial data is only assumed to attain its initial data $\ast$-weakly in $L^\infty(\R^n)$. What we must rule out is the possibility that $\norm{\theta(\cdot,t)}_{L^\infty(\R^n)}$ `jumps up' instantaneously at the initial time.  At a technical level, we require that the solution of the adjoint problem is strongly continuous at $t=0$. This is subtle because, for the adjoint problem, one can no longer easily justify a computation that would give energy equality. Instead, we defer to the fundamental solution estimates of Maekawa and Miura~\cite{MaekawaMiura-drift}, Theorem~1.2 and Remark~1.3, which ensure that nonetheless, there exists an adjoint solution that is strongly continuous in $L^2(\R^n)$. Let $T \in (0,1)$. Let $\psi_0 \in L^1 \cap L^\infty(\R^n)$. There exists $\psi \in C([0,T];L^1(\R^n)) \cap L^\infty_{t,x} \cap L^2_t H^{1/2}_x(Q_T)$ solving the adjoint equation
\begin{equation}
	-\p_t \psi - v \cdot \nabla \psi + \Lambda \psi = 0
\end{equation}
in $Q_T$ with final data $\psi(T) = \psi_0$. This solution can be chosen to satisfy, for all compact $K = K_1 \times K_2 \subset Q_T$, there exists $\gamma \in (0,1)$ such that $\psi \in L^\infty_t C^{1,\gamma}_x(K)$.\footnote{These H{\"o}lder estimates, which we derived before, can be justified at the level of the approximation procedure in~\cite{MaekawaMiura-drift}.}

 Let $0 < t_0 < t_1 < T$. Let $R, \varepsilon > 0$.
Let $\chi \in C^\infty_0(B_2)$ with $\chi \equiv 1$ on $B_1$ and $\chi_R = \chi(x/R)$. Let $\varphi^{t_0,t_1}_\varepsilon$ be a mollification of the indicator function $\mathbf{1}_{(t_0,t_1)}$ at scale $\varepsilon \ll 1$. We test the equation~\eqref{eq:basiclinearpde} against $\psi \chi_R \varphi^{t_0,t_1}_\varepsilon$. To simplify notation, we omit $R$, $\varepsilon$ and $t_0,t_1$. This gives
\begin{equation}
\begin{aligned}
	&\iint \underbrace{(\p_t + v \cdot \nabla + \Lambda) \theta}_{= 0 } \psi \chi \varphi \, dx \, dt \\
	&\quad = \iint \underbrace{(-\p_t - v \cdot \nabla + \Lambda) \psi}_{= 0} \theta \chi \varphi \, dx \, dt \\
	&\quad\quad + \iint (-\p_t - v \cdot \nabla )(\chi \varphi) \theta \psi +  \varphi \theta [\Lambda, \chi] \psi \, dx \, dt.
	\end{aligned}
\end{equation}
First, we send $\varepsilon \to 0^+$. For a.e. $t_0, t_1 \in (0,T)$, we have
\begin{equation}
	\iint -\p_t \varphi^{t_0,t_1}_\varepsilon \chi_R \theta \psi \to \int \chi_R \theta(x,t_1) \psi(x,t_1) \, dx - \int \chi_R \theta(x,t_0) \psi(x,t_0) \, dx.
\end{equation}
The other terms are well behaved, and we have
\begin{equation}
	\int \chi_R \theta(x,t_1) \psi(x,t_1) \, dx - \int \chi_R \theta(x,t_0) \psi(x,t_0) \, dx = \int_{t_0}^{t_1} \int_{\R^n}  v \cdot \nabla \chi_R \theta \psi +  \theta [\Lambda, \chi_R] \psi \, dx \, dt.
\end{equation}
Since $\theta(\cdot,t)$ is weak-$\ast$ continuous in $L^\infty(\R^n)$ on $[0,1]$ and $\psi(\cdot,t)$ is strongly continuous in $L^2(\R^n)$ on $[0,T]$, we may extend the above equality to all $t_0, t_1 \in [0,T]$. Second, we send $R \to +\infty$. Notice that
\begin{equation}
	\left| \int_{t_0}^{t_1} \int_{\R^n} (v \cdot \nabla \chi_R) \theta \psi \, dx \, dt \right| \les R^{-1+n/p} \norm{v}_{L^\infty_t X_p(Q_1)} \norm{\theta}_{L^\infty_{t,x}(Q_1)} \norm{\psi}_{L^\infty_t L^{p'}_x(Q_T)} \to 0,
\end{equation}
since $p > n$. Additionally,
\begin{equation}
	\label{eq:estimateihavetofix}
	\left| \iint \theta [\Lambda, \chi_R] \psi \, dx \, dt \right| \les R^{-1+n/p} \norm{\theta}_{L^\infty_{t,x}(Q_1)} \norm{\psi}_{L^\infty_t L^{p'}_x(Q_T)} \to 0.
\end{equation}
Here, $p'$ is the H{\"o}lder conjugate of $p$.
The estimate~\eqref{eq:estimateihavetofix} for $|x| \leq 2R$ follows from H{\"o}lder's inequality and the Calder{\'o}n commutator estimate $\norm{ [\Lambda, \chi_R] }_{L^q(\R^n) \to L^q(\R^n)} \les_q R^{-1}$ for all $q \in (1,+\infty)$.  To show~\eqref{eq:estimateihavetofix} for $|x| \geq 2R$, we require the pointwise bound
\begin{equation*}
	|[\Lambda,\chi_R] \psi(x,t)| = c_n \left| \int_{\R^n} \frac{\chi_R(y)}{|x-y|^{n+1}} \psi(y,t) \, dy \right| \les_p |x| ^{-(n+1)+n/p} \norm{\psi(\cdot,t)}_{L^{p'}(\R^n)}.
\end{equation*}
Combining the estimates in the regions $|x| \leq 2R$ and $|x| \geq 2R$ gives~\eqref{eq:estimateihavetofix}.
Next, the boundary terms in time are well behaved, since $\theta \in L^\infty(\R^n)$ and $\psi \in L^1(\R^n)$ for all $t \in [0,T]$. Hence,
\begin{equation}
	\int \theta(x,t_1) \psi(x,t_1) \, dx = \int \theta(x,t_0) \psi(x,t_0) \, dx.
\end{equation}
With $t_1 = T$ and $t_0 = 0$, we have
\begin{equation}
	\int \theta(x,T) \psi_0 \, dx = 0.
\end{equation}
Since $\psi_0$ was arbitrary, we have $\theta(\cdot,T) \equiv 0$. Finally, since $T$ was arbitrary, we have $\theta \equiv 0$.

(Continuity) The main work is to prove that the convergence~\eqref{eq:betterconv} holds and $\theta(\cdot,t) \wstar \theta_0$ as $t \to 0^+$. This is enough to demonstrate that $\theta$ belongs to our uniqueness class. We must estimate the time derivatives:
\begin{equation}
	\label{eq:timederivest}
	\p_t \theta^{(k)} = - \Lambda \theta^{(k)} - \div (v^{(k)} \theta^{(k)}) \in L^\infty_t (B^{-1}_{\infty,\infty})_x(Q_1) + L^\infty_t W^{-1,p}_x(B_R \times (0,1)),
\end{equation}
for all $R > 0$, with uniform bounds depending only on $M,B,N,V,p,R$, where $M$ denotes an upper bound for the norms of $\theta^{(k)}$ in $L^\infty(\R^n)$. Then~\eqref{eq:betterconv} follows from the \emph{a priori} $C^{1,\alpha_0}_x$ estimate~\eqref{eq:higherholder} and the Aubin--Lions lemma~\cite{simonforaubinlions}.\footnote{Recall that the spaces in Simon's refinement~\cite{simonforaubinlions} of the Aubin--Lions lemma are not required to be reflexive.} Moreover, the estimate~\eqref{eq:timederivest} on the time derivative gives that $\theta^{(k)}(\cdot,t) \wstar \theta(\cdot,t)$ in the sense of distributions for each $t \in [0,1]$. Similar arguments are well known from the Navier-Stokes literature. The estimate~\eqref{eq:timederivest} is valid also for $\theta$ and ensures that $\theta \: [0,1] \to L^\infty(\R^n)$ is weakly-$\ast$ continuous. This yields that the initial data $\theta_0$ is attained in the desired sense and completes the proof.
\end{proof}



We now specialize to the system~\eqref{eq:sqg} in dimension two.

\begin{lemma}[Higher spatial regularity of SQG solutions]
	\label{lem:higherspatialreg}
Let $\theta$ be the solution of Lemma~\ref{lem:Linftytheory} with dimension $n=2$. Assume that $v = \vec{R}^\perp \theta$, that is, $\theta$ is a solution of~\eqref{eq:sqg}, and $\norm{\theta_0}_{L^\infty(\R^2)} \leq M$. Then
\begin{equation}
	\sup_{t > 0} t^k \norm{\nabla_x^k \theta}_{L^\infty(\R^2)} \les_M 1
\end{equation}
for all $t \in (0,1]$ and integers $k \geq 0$.
\end{lemma}
\begin{proof}
We will demonstrate that, for all integers $k \geq 0$, we have
\begin{equation}
	\label{eq:higherholdersqg}
	\sup_{t > 0} t^{k+\alpha_0} [\theta(\cdot,t), v(\cdot,t)]_{C^{k,\alpha_0}_x(\R^2)} \les_{M} 1.
\end{equation}
Notice that $B \les M$ by $\vec{R} \: L^\infty(\R^2) \to \BMO(\R^2)$ and the maximum principle. Since $\vec{R} \: \dot C^{k,\alpha}(\R^2) \to \dot C^{k,\alpha}(\R^2)$ for all $\alpha \in (0,1)$ and $k \geq 0$, we have also $N \les_M 1$, and thus,~\eqref{eq:higherholdersqg} is verified for $k=0,1$ by the known H{\"o}lder and $C^{1,\alpha_0}_x$ estimates. We proceed by induction. Assuming that~\eqref{eq:higherholdersqg} holds for a given $k = k_0 \geq 1$, we show it for $k = k_0+1$. We differentiate the PDE:
\begin{equation}
	\p_t \nabla^{k_0} \theta + \Lambda \nabla^{k_0} \theta + v \cdot \nabla (\nabla^{k_0} \theta) = \mathcal{M}_1(\nabla v, \nabla^{k_0} \theta) + \cdots + \mathcal{M}_{k_0}(\nabla^{k_0} v, \nabla \theta),
\end{equation}
where the $\mathcal{M}_k$, $k=1,\hdots,k_0$, are certain bilinear operators that capture pointwise multiplications between the two terms. Let $\tilde{\theta} = \nabla^{k_0} \theta$ and $f = \sum_{k=1}^{k_0} \mathcal{M}_k$. Notice that the forcing term includes only derivatives of $v$, rather than $v$ itself. By interpolation and the induction hypothesis, we have
\begin{equation}
	\norm{\nabla^k v}_{L^\infty_t C^{\alpha_0}_x(\R^2 \times (1/2,1))} \les_{k,M} 1, \quad k = 1,\hdots,k_0.
\end{equation}
 Hence, $f \in L^\infty_t C^{\alpha_0}_x(\R^2 \times (1/2,1))$ with bound depending only on $M$ and $k_0$. At this point, we recall Silvestre's estimate~\eqref{eq:silvestreestimate}, which allowed a RHS $f$. The approach in Lemma~\ref{lem:higherreg}, which is easily modified to accommodate a forcing term, yields
\begin{equation}
	[\tilde{\theta}(\cdot,1)]_{C^{1,\alpha_0}(\R^2)} \les \left[ \norm{\tilde{\theta}}_{L^\infty_{t,x}(\R^2 \times (1/2,1))} + \norm{f}_{L^\infty_t C^{\alpha_0}_x(\R^2 \times (1/2,1))} \right] \les_M 1.
\end{equation}
Finally, the proof of~\eqref{eq:higherholdersqg} for $k=k_0+1$ is completed by scaling invariance.
\end{proof}

Since $v$ is only controlled in $L^\infty_t (X_p)_x(Q_1)$, the spacetime regularity is worse:

\begin{lemma}[Higher spacetime regularity of SQG solutions]
	\label{lem:higherspacetimereg}
Let $\theta$ be the solution as above, and additionally assume that $\theta$ is symmetric. Then
for all $t \in (0,1]$, integers $k, m \geq 0$, and all $p \in (2,+\infty)$,
\begin{equation}
	\label{eq:higherspacetimereginduction}
	t^{m+k} \norm{\p_t^m \nabla_x^k \theta}_{X_p} \les_{p,k,m,M} 1.
\end{equation}
\end{lemma}

The proof is based on elementary bootstrapping. The only difficulty is in the properties of the space $X_p$.

\begin{proof}[Proof sketch]
Whenever $\theta$ satisfies~\eqref{eq:higherspacetimereginduction} for a given $m,k,p$, the drift $v = \vec{R}^\perp \theta$ also satisfies~\eqref{eq:higherspacetimereginduction} because $\vec{R} \: X_p \to X_p$ in the symmetric case (Lemma~\ref{lem:riesz}).

The proof is by induction in $m$ with a sub-induction in $k$. We have shown the base case $m=0$ in Lemma~\ref{lem:higherspatialreg}, since $L^\infty(\R^n) \into X_p$.
Let $m_0 \geq 0$ be an integer. Assume that~\eqref{eq:higherspacetimereginduction} holds for all integers $0 \leq m \leq m_0$,
$k \geq 0$, and $p \in (2,+\infty)$. We apply $\p_t^{m_0}$ to~\eqref{eq:sqg} and have
\begin{equation}
	\label{eq:twoderivseq}
	\p_t^{m_0+1} \theta = -  \Lambda \p_t^{m_0} \theta - \p_t^{m_0} v \cdot \nabla \theta + \text{cross terms} - v \cdot \nabla \p_t^{m_0} \theta.
\end{equation}
Since $\Lambda = ( -\vec{R} \cdot) \nabla$ and $\vec{R} \: X_p \to X_p$, we have $t^{m_0+1} \Lambda \p_t^{m_0} \theta \in L^\infty_t X_p(Q_1)$ with the desired uniform bounds. Every other term on the RHS of~\eqref{eq:twoderivseq} is estimated by the induction hypothesis and H{\"o}lder's inequality in the spaces $X_{2p}$. This demonstrates the base case in the sub-induction in $k$, which is similar to the proof of Lemma~\ref{lem:higherspatialreg}.
\end{proof}

\section{Bounded solutions}


\subsection{DSS solutions with approximate constitutive law}


Let $\lambda > 1$ and $G$ be a symmetry group. Let $\theta_0 \in L^\infty(\R^2)$ be symmetric and $\lambda$-DSS with $\norm{\theta}_{L^\infty(\R^2)} \leq M$. 
Let $b:\R^2 \times \R_+ \to \R$ be a function satisfying the following standing assumptions:
\begin{itemize} 
\item $b(\cdot,t)$ is symmetric for every $t>0$, 
\item  $b\in L^2(B_1 \times (0,1))$, and
\item if $\tilde{\lambda} > 1$ and $\theta_0$ is $\tilde{\lambda}$-DSS, then $b$ is $\tilde{\lambda}$-DSS.
\end{itemize}
Let $\mathcal Y$ denote the class of functions $b$ satisfying the above properties. Then $\mathcal{Y}$ is a Banach space with norm $\|\cdot\|_{\mathcal Y}= \|\cdot\|_{L^2(B_1\times (0,1))}$. By $\lambda$-DSS scaling, we have
\begin{equation}
	\norm{b}_{L^2_{t,x}(K)} \les_{K} \norm{b}_{\mathcal{Y}}
\end{equation}
for all compact $K \subset \R^2 \times [0,+\infty)$. The $\lambda$-DSS property is used essentially in the construction. The purpose of mentioning $\tilde{\lambda}$ is so that, if the function is fully self-similar (rather than merely discretely self-similar), self-similarity is also kept in the construction.

Let $0 \leq \td \psi \in C^\infty_0(B_2)$ be a radial cut-off function with $\td\psi = 1$ on $B_1$. Let $\rho \gg 1$ and
\begin{equation}
	\psi(x,t)=\td \psi(x/t), \quad \psi_\rho(x,t)=\psi(x/\rho,t).
\end{equation} Let $0 \leq \eta \in C^\infty_0(\R^{2+1})$ with $\supp \eta \subset \{ |x| < 1, \, |t| < 1 \}$ and $\iint_{\R^{2+1}} \eta(x,t) \,dx \,dt= 1$. Let $0<\delta\ll 1$ and
\begin{equation}
	b_\delta(x,t)=  \int_{\R_+} \int_{\R^2} \frac 1 {(t\de)^3} \eta\bigg( \frac {x-y}{t\delta},\frac {t-s}{t\delta}  \bigg) b(y,s)   \,dy\,ds, \quad  b_{\de,\rho} (x,t) = b_\delta(x,t)\psi_\rho(x,t).
\end{equation}
Finally,
\begin{equation}
	v = \vec{R}^\perp(b_{\de,\rho}).
\end{equation}

The goal of this subsection is to prove the following proposition:

\begin{proposition}[DSS solutions to approximate problem]\label{prop.existence.SQG.local}
In the above notation, there exists a bounded, smooth, symmetric, and $\lambda$-DSS solution to 
\begin{equation}
	\label{eq:sqg.local}
\partial_t \th +\La \th  + \nb \cdot(v \th)=0,\quad v = \vec{R}^\perp ( \th_{\de,\rho}).
\end{equation}
If $\tilde{\lambda} > 1$ and $\theta_0$ is $\tilde{\lambda}$-DSS, then $\theta$ may be chosen to be $\tilde{\lambda}$-DSS as well.
\end{proposition}

The following lemma summarizes properties of the functions and vectors  we have introduced. We suppress the dependence on $\eta$ in the constants.

\begin{lemma}\label{lemma.drift.bounded}
In the above notation, $b_\de,\, b_{\de,\rho}$ and $v$ are symmetric and $\tilde{\lambda}$-DSS whenever $\theta_0$ is $\tilde{\lambda}$-DSS. For all integers $\ell, m \geq 0$, we have
\begin{equation}
	\label{eq:l2bounddrift}
	\sup_{t > 0} t^{\ell+m-1} \norm{\p_t^\ell \nabla^m_x v(\cdot,t)}_{L^2(\R^2)} \les_{\delta,\rho,\lambda,\ell,m} \norm{b}_{\mathcal{Y}}.
	\end{equation}
\end{lemma}

Immediately,~\eqref{eq:l2bounddrift} and interpolation give
\begin{equation}
	\label{eq:linftydriftbound}
	\sup_{t > 0} t^{\ell+m} \norm{\p_t^\ell \nabla^m_x v(\cdot,t)}_{L^\infty(\R^2)} \les_{\delta,\rho,\lambda,\ell,m} \norm{b}_{\mathcal{Y}}.
\end{equation}

\begin{proof}
The proof of symmetry and $\tilde{\lambda}$-DSS is by design of the mollification and truncation procedure. Regarding the $L^2$ bounds, by scaling invariance, it suffices to consider only $t\in [1,\lam]$, where the desired estimate~\eqref{eq:l2bounddrift} is obvious. \end{proof}

Treating $\th_0$ as fixed, we can define a map $\mathcal T_{\delta,\rho} \: \mathcal Y \to \mathcal Y$ in the following way: $\mathcal T_{\delta,\rho}(b\in \mathcal Y)=\th$, where $\th$ is the solution from Lemma~\ref{lem:Linftytheory} of the linear-in-$\theta$ PDE
\begin{equation}
	\label{eq:lineareq.b}
	\p_t \theta + \Lambda \theta + \div (  v \theta) = 0,\quad v = \vec{R}^\perp(b_{\de,\rho})
\end{equation}
with initial data $\theta_0$. 
Observe that $\theta \in \mathcal{Y}$ by the maximum principle.
 
We next  show that $\mathcal{T}_{\delta,\rho}$ has a fixed point using the Schauder fixed point theorem. To do so, we must show that $\mathcal{T}_{\delta,\rho}$ is continuous and that there exists a nonempty convex closed subset $\mathcal K$ of $\mathcal Y$ so that $\mathcal{T}_{\delta,\rho}:\mathcal K\to \mathcal K$ and  $\mathcal{T}_{\delta,\rho}(\mathcal K)$ is relatively compact, i.e.,~every subsequence in $\mathcal{T}_{\delta,\rho}(\mathcal K)$ has a convergent subsequence in $\mathcal Y$.  These properties are established below:

\begin{lemma}
The map $\mathcal{T}_{\delta,\rho}$ is continuous.
\end{lemma}
\begin{proof}
Let $( b^{(k)} )_{k \in \N} \subset \mathcal{Y}$ with $b^{(k)}\to b^{(\I)}$ in $\mathcal Y$. Let $(v^{(k)})_{k \in \N}$ and $v^{(\infty)}$ be the corresponding drifts, which satisfy $v^{(k)} \to v^{(\infty)}$ in $L^2_\loc(\R^2 \times \R_+)$ by $\lambda$-DSS scaling and the critical uniform bounds guaranteed by~Lemma~\ref{lemma.drift.bounded}. Hence, Lemma~\ref{lem:Linftytheory} gives that the corresponding solutions $\theta^{(k)}$ converge in $L^\infty_\loc(\R^2 \times \R_+)$ to the corresponding solution $\theta^{(\infty)}$. In particular, for all $S \in (0,1)$,
\begin{equation}
	\label{eq:proofofcontinuitymainstep}
	\int_0^{1} \int_{B_1} |\theta^{(k)} - \theta^{(\infty)}|^2 \, dx \,dt \leq 4 |B_1| S M^2 + o_{k \to +\infty}(1)
\end{equation}
where $o_{k \to +\infty}(1)$ may depend on $S$. This gives the desired convergence.
\end{proof}

Let $\mathcal K = \{  b\in \mathcal Y: \|b\|_{\mathcal Y}^2 \leq |B_1| M^2 \}$. This set is clearly nonempty, convex and closed.

\begin{lemma}
The map $\mathcal{T}_{\delta,\rho}$ maps $\mathcal{K}$ into $\mathcal{K}$.
\end{lemma}
\begin{proof}
Let $b \in \mathcal{Y}$.
By the maximum principle, we have
\begin{equation}
	\int_0^{1} \int_{B_1} |\theta|^2 \, dx \, dt \leq |B_1| M^2.
\end{equation}
\end{proof}

%
%
%

In fact, we have shown that $\mathcal{T}_{\delta,\rho}(\mathcal Y) \subset \mathcal K$. This is  simpler than  what happens for  the 3D Navier-Stokes equations~\cite{wolfchael2loc}.  

\begin{lemma}
Every sequence in $\mathcal{T}_{\delta,\rho}(\mathcal K)$ has a convergent subsequence in $\mathcal Y$.
\end{lemma}
This can be proven using a compactness argument common in the analysis of fluids. However, a more direct argument is available here due to regularity.

\begin{proof}
Let $(b^{(k)})_{k \in \N}$ be a sequence in $\mathcal{K}$ and $(\theta^{(k)})_{k \in \N}$ be the corresponding sequence of solutions. Hence, $v^{(k)}$ satisfies the critical bounds in Lemma~\ref{lemma.drift.bounded} uniformly in $k$, and upon passing to a subsequence, we have that $v^{(k)} \wstar v$ in the sense of distributions. Then Lemma~\ref{lem:Linftytheory} ensures the desired convergence $\theta^{(k)} \to \theta$ in $L^\infty_\loc(\R^2 \times \R_+)$. Then we use~\eqref{eq:proofofcontinuitymainstep} as in the proof of continuity to conclude the convergence in $\mathcal{Y}$.

\end{proof}

\subsection{Proof of Theorem~\ref{thm:boundedsolutions}}
 

 Let $\theta_0 \in L^\infty(\R^2)$ be symmetric with $\norm{\theta_0}_{L^\infty(\R^2)} \leq M$.

\emph{1. General solutions}.
Consider a sequence $(\theta^{(k)}_0)_{k \in \N}$ of symmetric functions in $C^\infty_0(\R^n)$ with $\norm{\theta^{(k)}_0}_{L^\infty(\R^2)} \leq M$ and $\theta^{(k)}_0 \wstar \theta_0$ in $L^\infty(\R^2)$ as $k \to +\infty$. Such a sequence is obtained by mollifying $\theta_0$ and applying cut-offs. Let $\theta^{(k)}$ denote the unique global solution of~\eqref{eq:sqg}, belonging to $C([0,+\infty);H^m(\R^2))$ for all $m \in \N$ with initial data $\theta^{(k)}_0$. The existence of these solutions is guaranteed by Caffarelli and Vasseur~\cite{CV} or Constantin and Vicol~\cite{ConstVicolGAFA}.\footnote{The short-time well-posedness can be developed by energy estimates, and the two mentioned sources guarantee that the solution may be continued globally.} By virtue of its uniqueness, $\theta^{(k)}$ is symmetric. Moreover, the solutions satisfy the necessary estimates (which we also mention below) to apply the compactness in Lemma~\ref{lem:Linftytheory}. Upon passing to a subsequence, $\theta^{(k)}$ converges to the desired global solution of~\eqref{eq:sqg}.

\emph{2. Self-similar solutions}.
Assume further that $\theta_0$ is $\lambda$-DSS. We require a solution method which respects the self-similarity, and therefore, we no longer mollify the initial data. Let $1 \leq \rho^{(k)} \to +\infty$ and $\delta^{(k)} \to 0$ with $0 < \delta^{(k)} \leq 1$. Let $\theta^{(k)}$ be the corresponding solutions of the approximate SQG~\eqref{eq:sqg.local} from Proposition~\ref{prop.existence.SQG.local}. Since $\th^{(k)}$ are generated by the map $\mathcal{T}_{\delta^{(k)},\rho^{(k)}}$, they are the solutions of Lemma~\ref{lem:Linftytheory}. In particular, we have the maximum principle $\sup_{k} \|\th^{(k)}\|_{L^\I ( \R^2\times \R_+)}\leq \|\th_0\|_{L^\I(\R^2)}$. By properties of the Riesz transforms, we have
\[ \| v^{(k)} \|_{L^\I (\R_+;X_p)} = \| v^{(k)} \|_{L^\I (\R_+;X_{p,\osc})}  \les_p  \|v^{(k)}\|_{L^\I (\R_+;\BMO(\R^2))} \les \|\th_0\|_{L^\I(\R^2)},\]
where we used symmetry and the fact that the localized and smoothed drift velocities in the definitions of $v^{(k)}$ can be bounded by $\th^{(k)}$ in an obvious way. Hence, we have the uniform-in-$k$ H\"older estimate~\eqref{eq:initialholder} and $C^{1,\alpha_0}_x$ estimate~\eqref{eq:higherholder} depending only on $\norm{\theta_0}_{L^\infty(\R^2)}$. We are now in a position to apply the compactness part of Lemma~\ref{lem:Linftytheory}. Let $v = \vec{R}^\perp \th$. The convergence properties of $\th^{(k)}$ are strong enough to ensure that $v^{(k)}\wstar v$ in the sense of distributions and that $\th$ solves \eqref{eq:sqg}. The compactness part of Lemma~\ref{lem:Linftytheory} also guarantees that $\theta$ attains its initial data $\theta_0$ in the sense of~\eqref{eq:attaininitialdatabounded}. This completes the existence proof.

\emph{3. Smoothness}. The advertised smoothness follows from Lemmas~\ref{lem:higherspatialreg} and~\ref{lem:higherspacetimereg}.

\emph{4. Weak-$\ast$ stability}. This follows from the compactness part of Lemma~\ref{lem:Linftytheory} and Remark~\ref{rmk:riesztransformsandconvergence}, which handles the convergence of the Riesz transforms.

\section{Regular solutions}

In this section, we demonstrate Theorem~\ref{thm:moreregularsolutions}. Recall the definitions of $\dot Y^{\alpha}$ and $\mathbb{Y}^\alpha$ in~\eqref{eq:ydotdef} and~\eqref{eq:ybbdef}.

Let $\theta^{\lin} = e^{-t\Lambda} \theta_0$. Our goal is to estimate $\psi = \theta - \theta^{\lin}$. The equation satisfied by $\psi$ is
\begin{equation}
\begin{aligned}
	\p_t \psi + \Lambda \psi + \vec{R}^\perp \theta \cdot \nabla \psi + \vec{R}^\perp \psi \cdot \nabla \theta^\lin &= - \vec{R}^\perp \theta^\lin \cdot \nabla \theta^\lin \\
	\psi(\cdot,0) &= 0.
	\end{aligned}
\end{equation}


\begin{proof}[Proof of \emph{a priori} estimates]
We use a Calder{\'o}n-type splitting into subcritical and supercritical parts. Let $\varphi \in C^\infty(B_1)$ be a radial test function with $\varphi \equiv 1$ on $B_{1/2}$. Let
\begin{equation}
	\theta_0 = \tilde{\theta}_0 + \bar{\theta}_0,
\end{equation}
where
\begin{equation}
	\tilde{\theta}_0 = \varphi \theta_0, \quad
	\bar{\theta}_0 = (1-\varphi) \theta_0.
\end{equation}
Then, for all $p \in [1,+\infty]$, we have
\begin{equation}
	\label{eq:thetaendupwhere}
	\norm{\bar{\theta}_0}_{\mathbb{Y}^\alpha} + \norm{\tilde{\theta_0}}_{\mathbb{Y}^\alpha}  + \norm{\bar{\theta}_0}_{C^\alpha} + \norm{\tilde{\theta_0}}_{L^p} \les_p M,
\end{equation}
where we suppress the dependence on $\varphi$. Here, $\tilde{\theta}_0$ is the supercritical part, and $\bar{\theta}_0$ is the subcritical part. The main non-trivial assertion in~\eqref{eq:thetaendupwhere} is that $\vec{R} \bar{\theta}_0, \vec{R} \tilde{\theta}_0 \in L^\infty$, which we establish below. Specifically, we focus on estimating the following commutator in $L^\infty(\R^2)$:
\begin{equation}
	[\vec{R},\varphi] \theta_0 = \pv \int_{\R^n} [K(x,y) - K(0,y) \mathbf{1}_{|y| \geq 2}] (\varphi(x) - \varphi(y)) \theta_0(y) \, dy
\end{equation}
where $K(x,y) = c_n (x-y)/|x-y|^{n+1}$. When $x \in B_1$ and $y \in B_2$, we write $\varphi(x) - \varphi(y) = \nabla \varphi(x) \cdot (x-y) + O(|x-y|^2)$ and use that the new kernel is summable in $|y| \leq 2$. 
When $x \in B_1$ and $|y| \geq 2$, we have
\begin{equation}
	\left\lvert \varphi(x) \int_{\R^n \setminus B_2} [K(x,y) - K(0,y)] \theta_0(y) \, dy \right\rvert \les \norm{\theta_0}_{L^\infty(\R^2 \setminus B_2)},
\end{equation}
since the above kernel is summable over $|y| \geq 2$. Finally, when $|x| \geq 1$, we have
\begin{equation}
	\left\lvert \int_{B_1} K(x,y) \varphi(y) \theta_0(y) \, dy \right\rvert \les |x|^{-2} \norm{\theta_0}_{L^1(B_1)}.
\end{equation}
Combining the above three estimates gives the proof.

Next, we decompose the solution as
\begin{equation}
	\theta = \underbrace{e^{-t\Lambda} \bar{\theta}_0}_{\bar{\theta}} + \tilde{\theta}.
\end{equation}
Hence,
\begin{equation}
	\psi = \underbrace{\bar{\theta} - \theta^\lin}_{ - e^{-t\Lambda} \tilde{\theta}_0} + \tilde{\theta}.
\end{equation}
For the linear evolution $\bar{\theta}$ of the subcritical data $\bar{\theta}_0$, we have
\begin{equation}
	\sup_t \, \norm{\nabla \bar{\theta}(\cdot,t)}_{L^{2,\infty}} + \sup_t t^{1-\alpha} \norm{\nabla \bar{\theta}(\cdot,t)}_{L^\infty} \les_\alpha M.
\end{equation}
Interpolation yields
\begin{equation}
	\label{eq:barthetaest}
	\sup_t t^{ (1-\alpha) \left( 1-\frac{2}{p} \right) } \norm{\nabla \bar{\theta}(\cdot,t)}_{L^p} \les_\alpha M
\end{equation}
for all $p \in (2,+\infty]$, and, in particular,
\begin{equation}
	\label{eq:wegottahide}
	\norm{\nabla \bar{\theta}}_{L^1_t L^p_x(Q_T)} \les_{\alpha,p} T^{\frac{2}{p}+\alpha-\frac{2\alpha}{p}} M.
\end{equation}
The above exponent is positive.
For the linear evolution of the supercritical data $\tilde{\theta}_0$, we have
\begin{equation}
\label{eq:supercriticalevol}
	\frac{1}{p} \sup_{t \in (0,T)} \int_{\R^2} |e^{-t\Lambda} \tilde{\theta}_0|^p(x,t) \, dx + c_p \int_{0}^{T} \int_{\R^2} |\Lambda^{\frac{1}{2}} |e^{-t\Lambda} \tilde{\theta}_0|^{\frac{p}{2}}|^2 \, dx \, dt \les_p M.
\end{equation}



The equation satisfied by $\tilde{\theta}$ is
\begin{equation}
	\begin{aligned}
		\p_t \tilde{\theta} + \Lambda \tilde{\theta} + \vec{R}^\perp \theta \cdot \nabla \tilde{\theta} + \vec{R}^\perp \tilde{\theta} \cdot \nabla \bar{\theta} &= - \vec{R}^\perp \bar{\theta} \cdot \nabla \bar{\theta} \\
		\tilde{\theta}(\cdot,0) &= \tilde{\theta}_0.
		\end{aligned}
\end{equation}
We wish to perform energy estimates. Interestingly, it is unnecessary to ask beforehand that $\psi, \tilde{\theta} \in C([0,1];L^p(\R^2))$. This is due to the criticality of the linear-in-$\theta$ PDE~\eqref{eq:basiclinearpde}: $\theta$ is the unique bounded, smooth, symmetric solution, see Lemma~\ref{lem:Linftytheory}, and one may justify the following calculations through an approximation procedure for this unique solution. Multiplying by $|\tilde{\theta}|^{p-2} \tilde{\theta}$ and integrating by parts, we have
\begin{equation}
\begin{aligned}
	 &\frac{1}{p} \int_{\R^2} |\tilde{\theta}|^p(x,t_2) \, dx + c_p \int_{t_1}^{t_2} \int_{\R^2} |\Lambda^{\frac{1}{2}} |\tilde{\theta}|^{\frac{p}{2}}|^2 \, dx \, dt \\
	 &\quad \leq \int_{t_1}^{t_2} \int_{\R^2} \left\lvert \vec{R}^\perp \tilde{\theta} \cdot \nabla \bar{\theta} |\tilde{\theta}|^{p-2} \tilde{\theta}\right\rvert + \left\lvert \vec{R}^\perp \bar{\theta} \cdot \nabla \bar{\theta} |\tilde{\theta}^{p-2}| \tilde{\theta} \right\rvert \, dx \, dt + \frac{1}{p} \int_{\R^2} |\tilde{\theta}|^p(x,t_1) \, dx\\
	&\quad \les_p \norm{\tilde{\theta}}_{L^\infty_t L^p_x(Q_T)}^p \norm{\nabla \bar{\theta}}_{L^1_t L^\infty_x(Q_T)} + \norm{\tilde{\theta}}_{L^\infty_t L^p_x}^{p-1} \norm{\vec{R}^\perp \bar{\theta}}_{L^\infty_{t,x}(Q_T)} \norm{\nabla \bar{\theta}}_{L^1_t L^p_x(Q_T)} + \int_{\R^2} |\tilde{\theta}|^p(x,t_1) \, dx
	\end{aligned}
\end{equation}
for a.e. $t_1, t_2 \in (0,+\infty)$.
Taking $t_1 \to 0^+$ and using Young's inequality on the second term on the RHS, we have
\begin{equation}
\begin{aligned}
	&\sup_{t \in (0,T)} \frac{1}{p} \int_{\R^2} |\tilde{\theta}|^p(x,t) \, dx + c_p \int_{0}^{T} \int_{\R^2} |\Lambda^{\frac{1}{2}} |\tilde{\theta}|^{\frac{p}{2}}|^2 \, dx \, dt \\
	&\quad \les_p \norm{\tilde{\theta}}_{L^\infty_t L^p_x(Q_T)}^p \norm{\nabla \bar{\theta}}_{L^1_t L^\infty_x(Q_T)} + \norm{\vec{R}^\perp \bar{\theta}}_{L^\infty_{t,x}(Q_T)}^p \norm{\nabla \bar{\theta}}_{L^1_t L^p_x(Q_T)}^p + \norm{\tilde{\theta}_0}_{L^p(\R^2)}^p.
	\end{aligned}
\end{equation}
Since $\bar{\theta}$ is subcritical, we use~\eqref{eq:wegottahide} (with $p=+\infty$) to absorb the first term on the RHS into the LHS when $T \ll_{\alpha,p,M} 1$. The second term on the RHS is estimated by~\eqref{eq:barthetaest}--\eqref{eq:wegottahide}. This gives
\begin{equation}
	\frac{1}{p} \sup_{t \in (0,\bar{T})} \int |\tilde{\theta}|^p(\cdot,t) + c_p \int_{0}^{\bar{T}} \int |\Lambda^{\frac{1}{2}} |\tilde{\theta}|^{\frac{p}{2}}|^2 \les_{M,p} 1.
\end{equation}
with $\bar{T} = \bar{T}(\alpha,p,M) \leq 1$.
Combining these with the estimates~\eqref{eq:supercriticalevol} for $e^{-t\Lambda} \tilde{\theta}_0$ and the decomposition $\psi = \tilde{\theta} - e^{-t\Lambda} \tilde{\theta}_0$, we have
\begin{equation}
	\frac{1}{p} \sup_{t \in (0,\bar{T})} \int |\psi|^p(x,t) \, dx +  c_p \int_{0}^{\bar{T}} \int |\Lambda^{\frac{1}{2}} |\psi|^{\frac{p}{2}}|^2 \, dx \,dt \les_{M,p} 1.
\end{equation}
The estimates for arbitrary times follow from scaling invariance of the norms.
\end{proof}

It is likely that the solution can be shown to belong to $L^{2,\infty}(\R^2)$ at each time as well.

\begin{proof}[Proof of higher regularity]
The main obstacle is to show that $v = \vec{R}^\perp \theta$ belongs to $L^\infty_{t,x}(\R^2 \times (1/2,1))$ with estimates depending only on $M$. Scaling invariance automatically extends the estimate to $L^\infty_{t,x}(\R^2 \times \R_+)$. The proof is completed by suitably adjusting the proof of Lemma~\ref{lem:higherspacetimereg}.

We need only estimates on $\vec{R}^\perp \psi$, since $e^{-t\Lambda} \vec{R}^\perp  \theta_0$ is well understood and $\vec{R}^\perp$ commutes with the semigroup. Because $\vec{R}^\perp \: L^4(\R^2) \to L^4(\R^2)$, we know that $\vec{R}^\perp \psi \in L^\infty_t L^4_x(Q_1)$ with estimates depending only on $M$. We exploit the identity
\begin{equation}
	\label{eq:psiidentity}
	\nabla \vec{R}^\perp \theta(\cdot,t) = \nabla e^{-t\Lambda} \vec{R}^\perp \theta_0 + \nabla \vec{R}^\perp \psi
\end{equation}
and the estimates
\begin{equation}
	\label{eq:thisthingwewilluse}
	\sup_{t \in (0,1)} t \norm{\nabla \vec{R}^\perp \theta(\cdot,t)}_{L^\infty(\R^2)} + t \norm{\nabla e^{-t\Lambda} \vec{R}^\perp \theta_0}_{L^\infty(\R^2)} \les_M 1.
\end{equation}
Let us explain how to estimate the first term in~\eqref{eq:thisthingwewilluse}. Theorem~\ref{thm:boundedsolutions} guarantees that $\sup_{t \in (0,1)} \norm{\theta(\cdot,t)}_{L^\infty(\R^2)} + t^2 \norm{\nabla^2 \theta(\cdot,t)}_{L^\infty(\R^2)} \les_M 1$, which implies that $\sup_{t \in (0,1)} \norm{\theta(\cdot,t)}_{\dot B^{0}_{\infty,\infty}(\R^2)} + t^2 \norm{\theta(\cdot,t)}_{\dot B^2_{\infty,\infty}(\R^2)} \les_M 1$. After interpolation, we have $\sup_{t \in (0,1)} t \norm{\theta(\cdot,t)}_{\dot B^1_{\infty,1}} \les_M 1$, and the Riesz transform is bounded on $\dot B^1_{\infty,1}(\R^2)$ (see Proposition~2.30 in~\cite{bahourichemindanchin}). This justifies~\eqref{eq:thisthingwewilluse}. The identity~\eqref{eq:psiidentity}, the estimate~\eqref{eq:thisthingwewilluse}, and the triangle inequality imply that $\sup_{t \in (0,1)} t \norm{\nabla \vec{R}^\perp \psi(\cdot,t)}_{L^\infty(\R^2)} \les_M 1$. Finally, interpolation between the estimates for $\vec{R}^\perp \psi$ in $L^\infty_t L^4_x(\R^2 \times (1/2,1))$ and $L^\infty_t \dot W^{1,\infty}_x(\R^2 \times (1/2,1))$ gives that $\vec{R}^\perp \psi \in L^\infty_{t,x}(\R^2 \times (1/2,1))$ with estimates depending only on $M$, as desired.
\end{proof}

\begin{proof}[Proof of uniqueness]
Let $\theta_1$, $\theta_2$ be two solutions with the same initial data. Let $\psi_k = \theta_k - e^{-t\Lambda} \theta_0$, $k = 1,2$. Let $f = \theta_1 - \theta_2 = \psi_1 - \psi_2$. Then $f$ satisfies the equation
\begin{equation}
	\begin{aligned}
		\p_t f + \Lambda f + \vec{R}^\perp \theta_1 \cdot \nabla f + \vec{R}^\perp f \cdot \nabla \theta_2 &= 0 \\
		f(\cdot,0) &= 0.
		\end{aligned}
		\end{equation}
		Energy estimates with $p=4$ yield
		\begin{equation}
			\label{eq:energyestp4}
\begin{aligned}
	 &\frac{1}{4} \int |f|^4(x,t_1) \, dx + c_4 \int_{t_0}^{t_1} \int |\Lambda^{\frac{1}{2}} |f|^2|^2 \, dx \, dt \\
	 &\quad \leq \int_{t_0}^{t_1} \int \left\lvert \vec{R}^\perp f \cdot \nabla \theta_2 |f|^2 f \right\rvert \, dx \, dt + \frac{1}{4} \int |f|^4(x,t_0) \, dx
	\end{aligned}
\end{equation}
for a.e. $0 < t_0 < t_1 < +\infty$. Since $\norm{f(\cdot,t)}_{L^4(\R^2)} \les t^{1/2}$, we may include $t_0 = 0$ and write $T = t_1$. We now estimate
\begin{equation}
	\label{eq:estimateondrifttermp4}
\begin{aligned}
	\int_{0}^{T} \int \left\lvert \vec{R}^\perp f \cdot \nabla \theta_2 |f|^2 f \right\rvert \, dx \, dt &\leq \int_0^T t \, dt\times  \norm{t^{-1/2} f}_{L^\infty_t L^4_x(Q_T)}^4 \norm{t \nabla \theta_2}_{L^\infty(\R^2)} \\
	&\leq C(M) T^2 \norm{t^{-1/2} f}_{L^\infty_t L^4_x(Q_T)}^4,
	\end{aligned}
\end{equation}
where $C(M) \to 0$ as $M \to 0^+$. Inserting~\eqref{eq:estimateondrifttermp4} into~\eqref{eq:energyestp4} and performing basic manipulations, we have
\begin{equation}
	\norm{t^{-1/2} f}_{L^\infty_t L^4_x(Q_T)} \leq C(M) \norm{t^{-1/2} f}_{L^\infty_t L^4_x(Q_T)}.
\end{equation}
When $M \ll 1$, we have that $f \equiv 0$. This completes the proof of uniqueness.
\end{proof}

\section{Local energy estimates}
\label{sec:localenergy}

This section is devoted to local energy estimates for the non-local drift-diffusion equation~\eqref{eq:basiclinearpde} from Section~\ref{sec:linftytheory}. The PDE is
\begin{equation}\label{eq:linear}
	\p_t \theta + v \cdot \nabla \theta + \Lambda \theta = 0,
\end{equation}
where $\div v = 0$.  We will usually assume convergence to the initial data $\th_0$ in the following sense: For every compact set $K\subset \R^n$, 
\EQ{\label{eq:attaininitialdata}
\lim_{t\to 0^+} \|  \th(\cdot,t)-\th_0\|_{L^2(K)}=0.
}
We cast our energy estimates in terms of function spaces $A_T$ and $E_T$, the definitions of which we presently recall. Let $T \in (0,+\infty]$ and $f \in L^1_\loc(Q_T)$ satisfying $f(\cdot,t) \in X_1$ for a.e. $t \in (0,T)$. Define
\begin{equation}
	\norm{f}_{A_T^{R_0}} := \esssup_{t \in (0,T)} \norm{f(\cdot,t)}_{X_2^{R_0}}.
\end{equation}
Then 
 $\Lambda f(\cdot,t)$ is well defined as a tempered distribution for a.e. $t \in (0,T)$. If additionally $\Lambda f \in L^1_\loc(Q_T)$, then
\begin{equation}
	\norm{f}_{E_T^{R_0}}^2 := \sup_{R \geq R_0} \frac{1}{R^n}  \int_0^T \int_{B_R} |\Lambda^{1/2} f|^2 \, dx \, dt.
\end{equation}
The space $A_T^{R_0}$ (resp. $E_T^{R_0}$) is defined by the property that $\norm{f}_{A_T^{R_0}} < +\infty$ (resp. $\norm{f}_{E_T^{R_0}} < +\infty$). When $R_0 = 1$, we write simply $A_T$ (resp. $E_T$). There is an appropriate notion of weak-$\ast$ convergence in these spaces in which the norms are lower semi-continuous.

To begin, we require a few facts about fractional operators and Sobolev spaces. First, $\kappa \in (2,4]$ denotes the exponent corresponding to the Sobolev embedding $\dot H^{1/2}(\R^n) \into L^\kappa(\R^n)$. That is,
\begin{equation}
	\label{eq:sobolevembedding}
	\norm{g}_{L^\kappa(\R^n)} \les \norm{\Lambda^{1/2} g}_{L^2(\R^n)},
\end{equation}
where
\begin{equation}
	\frac{n}{2} = \frac{n}{\kappa} + \frac{1}{2}.
\end{equation}
In particular, $n$ is the H{\"o}lder conjugate of $\kappa/2$.\footnote{The embedding~\eqref{eq:sobolevembedding} is valid, for example, when $g \in L^1(\R^n)$ is compactly supported and the RHS of~\eqref{eq:sobolevembedding} is finite. This is the context in which we use it.} Second, we have


\begin{lemma}[Commutator estimate]
\label{lem:muhcommutatorest}
 Let $f \in X_1$ with $\Lambda f \in L^1_\loc(\R^n)$.
Let $\phi \in C^\infty_0(B_1)$ and $\phi_R(x) = \phi(x/R)$ with $R \geq 1$. Then
\begin{equation}
	\label{eq:commutatorest}
	\norm{[\Lambda^{1/2},\phi_R] f}_{L^2(\R^n)} \les_\phi R^{(n-1)/2} \norm{f}_{X_2}.
\end{equation}
\end{lemma}
We use the convention $[A,B] = AB - BA$. 
\begin{proof}
It suffices to consider $R = 1$. When $x \in B_2$, we have
\begin{equation}
\begin{aligned}
	\left[ \Lambda^{1/2}, \phi \right] f(x) &= {\rm pv} \int_{\R^n} K(x,y) (\phi(x) - \phi(y)) f(y) \, dy \\
	&= \underbrace{\nabla \phi(x) \cdot \int_{B_2} \frac{x-y}{|x-y|^{n+1/2}} f(y) + O(|x-y|^{-n+3/2}) f(y) \, dy}_{g_1} \\
	&\quad + \underbrace{\int_{\R^n \setminus B_2} K(x,y) \phi(x) f(y) \, dy}_{g_2}.
	\end{aligned}
\end{equation}
We clearly have $\norm{g_1}_{L^2(B_2)} \les \norm{f}_{L^2(B_2)}$, since the kernel defining $g_1$ is integrable. Regarding $g_2$, we have
\begin{equation}
	\label{eq:g2part}
\begin{aligned}
	|g_2(x)| &\les \sum_{k=1}^{+\infty} 2^{-nk-k/2} \norm{P_k f}_{L^1(\R^n)} \\
	&\les \norm{f}_{L^2(B_2)} + \sum_{k=1}^{+\infty} 2^{-nk/2-k/2} \norm{f}_{L^2(B_{2^{k+1}} \setminus B_{2^k})} \\
	&\les \norm{f}_{L^2(B_2)} + \sum_{k=1}^{+\infty} 2^{-k/2} \norm{f}_{X_2} \\
	&\les \norm{f}_{X_2},
	\end{aligned}
\end{equation}
and $\norm{g_2}_{L^2(B_2)} \les \norm{f}_{X_2}$.
When $x \not\in B_2$, we have
\begin{equation}
	|[\Lambda^{1/2},\phi] f(x)| = \left\lvert \int_{B_1} K(x,y) \phi(y) f(y) \right\rvert \les |x|^{-n-1/2} \norm{f}_{L^2(B_1)},
\end{equation}
which is estimated in $L^2(\R^n \setminus B_2)$ by $\norm{f}_{L^2(B_1)}$. This completes the proof.
\end{proof}
In particular, when $\phi \equiv 1$ on $B_{1/2}$, we have
\begin{equation}
	\label{eq:sobolevembeddingcontinued}
	\norm{f}_{L^\kappa(B_R)} \overset{\eqref{eq:sobolevembedding}}{\les} \norm{\Lambda^{1/2} (f \phi_{2R})}_{L^2(B_{2R})} \overset{\eqref{eq:commutatorest}}{\les}_\phi \norm{\Lambda^{1/2} f}_{L^2(B_{2R})} + R^{(n-1)/2} \norm{f}_{X_2}.
\end{equation}

\begin{remark}[Local energy inequality]\label{rmk:localenergyinequality} We will use a version of the local energy inequality in what follows. In our applications, the computations to produce such an inequality are justified by smoothness. However, the local energy inequality can be established under weaker conditions which we now identify.
Let $T > 0$. Let $\theta \in A_T \cap E_T$ with $\p_t \theta \in (L^2_t H^{-1/2}_x)_\loc(Q_T)$. Let  $v \in (L^\infty_t C^{1/2}_x)_\loc(Q_T)$ be a divergence-free vector field satisfying
\begin{equation}
	\label{eq:lineareq}
	\left[ \p_t \theta + \Lambda \theta + \div (v \theta) \right] \cdot \theta \leq 0
\end{equation}
in the sense of distributions (with non-negative test functions) on $Q_T$. This is akin to the global energy situation in Lemma~\ref{lem:l2theory}, although we allow inequality in~\eqref{eq:lineareq} to account for the possibility that $\theta = |\eta|^{p/2 -1}\eta$, where $\eta$ is a smooth solution of~\eqref{eq:sqg}. In the proof of Lemma~\ref{lem:aprioriestimates}, it is shown, using the commutator estimate in Lemma~\ref{lem:muhcommutatorest}, that the term $\theta \Lambda \theta$ is well defined as a distribution under these assumptions. Then $\theta$ satisfies the local energy inequality
\begin{equation}
	\label{eq:localenergyeq}
	- \frac{1}{2} \iint |\theta|^2 \p_t \phi \, dx \,dt + \la \Lambda \theta, \theta \phi \ra \leq \frac{1}{2} \iint |\theta|^2  v \cdot \nabla \phi \, dx \,dt
\end{equation}
for every non-negative $\phi \in C^\infty_0(Q_T)$. Let $\theta_0 \in X_2$. If~\eqref{eq:attaininitialdata} is satisfied, then we furthermore have
\begin{equation}
	\label{eq:localenergyineqtimeslice}
\begin{aligned}
	&\int |\theta(x,t_2)|^2 \phi(x) \, dx + 2 \int_{t_1}^{t_2} \la \Lambda \theta(\cdot,t), \theta(\cdot,t) \phi \ra \, dt \leq \int |\theta(x,t_1)|^2 \phi(x) \, dx \\
	&\quad+ \int_{t_1}^{t_2} \int_{\R^n} |\theta|^2 v \cdot \nabla \phi \, dx \, dt
	\end{aligned}
\end{equation}
for a.e. $t_1 \in [0,T)$, including $t_1 = 0$, every $t_2 \in (t_1,T]$, and every non-negative $\phi \in C^\infty_0(\R^n)$.

 
\end{remark}

 \begin{lemma}[Local energy estimates]
\label{lem:aprioriestimates}
 
Let $\theta_0 \in X_2$ and $v$ be a divergence-free vector field satisfying 
\begin{equation}
	 \norm{v}_{L^\infty_t (X_p)_x(\R^n \times \R_+)} \leq V,
\end{equation}
where $p \in (n,+\infty)$ (or $p = n$ and $V \ll 1$) and $\th\in A_T\cap E_T$ for every $T > 0$. 
Suppose that the local energy inequality
\begin{equation}\label{eq:linearinequality}
	\p_t |\theta|^2 + 2 \theta \Lambda \theta + \div (v |\theta|^2 ) \leq 0
\end{equation}
holds 
in the sense of distributions (with non-negative test functions), and the initial data $\theta_0$ is attained in the sense of~\eqref{eq:attaininitialdata}. 
Then 
\begin{equation}
	\label{eq:energyestimatesuitableforGronwall}
	\norm{\theta}_{A_T}^2 + \norm{\theta}_{E_T}^2 \les  \norm{\theta_0}^2_{X_2}  +\int_0^T \|\th\|_{A_t}^2\,dt+  C_p  \int_0^T \|v\|_{L^\I (0,t;X_p)}^{2/(1-\tau)} \|\th\|_{A_t}^2 \,dt,
\end{equation}
where $\tau = \tau(p) \in (0,1]$ is defined in~\eqref{eq:taudef}.
Consequently, there exists $T_0 = T_0(p,V)  \in (0,1]$ such that
\begin{equation}
	\label{eq:energyestimatelesssuitable}
	\norm{\theta}_{A_{T_0}} + \norm{\theta}_{E_{T_0}} \les_p \norm{\theta_0}_{X_2}.
\end{equation}
Hence, by scaling invariance, we have, for all $T \geq T_0$,
\begin{equation}
	\norm{\theta}_{A_{T}^{T/T_0}} + \norm{\theta}_{E_T^{T/T_0}} \les_p \norm{\theta_0}_{X_2^{T/T_0}} \les_p \norm{\theta_0}_{X_2}.
\end{equation}
\end{lemma}

The assumption $p \geq n$ implies that $v |\theta|^2$ belongs to $L^1_\loc$. 
 
We adopt the notation $Q_{R,T} = B_R \times (0,T)$.

\begin{proof}
Fix $R\geq 1$. Fix $0 \leq \phi \in C^\infty_0(B_1)$ with $\phi \equiv 1$ on $B_{3/4}$. Let $\phi_R = \phi(\cdot/R)$. (Constants below may implicitly depend on $\phi$.) 

For the diffusive term, we use the following trick from~\cite{Lazar1}:
\begin{equation}
\begin{aligned}
\la \La \th ,\th\phi_R \ra &= \int |\La^{1/2}\th|^2\phi_R\,dx + \int \La^{1/2}\th [ \La^{1/2},\phi_R  ]\th\,dx \\
&= \int |\La^{1/2}\th|^2\phi_R\,dx + \int \phi_{2R }\La^{1/2}\th [ \La^{1/2},\phi_R  ]\th\,dx \\
&\quad +\int (1-\phi_{2R })\La^{1/2}\th [ \La^{1/2},\phi_R  ]\th\,dx.
\end{aligned}
\end{equation}
By the commutator estimate in Lemma~\ref{lem:muhcommutatorest}, we have, for a.e. $t \in (0,T)$,
\begin{equation}
\begin{aligned}
 \int |\phi_{2R }\La^{1/2}\th [ \La^{1/2},\phi_R  ]\th |\,dx  &\les \norm{\Lambda^{1/2} \theta}_{L^2(B_{2R})} \norm{[\Lambda^{1/2}, \phi_R]}_{L^2(B_{2R})} \\
& \overset{\eqref{eq:commutatorest}}{\les}  R^{(n-1)/2}  \norm{\Lambda^{1/2} \theta}_{L^2(B_{2R})}  \norm{\theta}_{X_2}. \\
\end{aligned}
\end{equation}
Integrating in time, we have
\begin{equation}
	\int_0^T \int |\phi_{2R }\La^{1/2}\th [ \La^{1/2},\phi_R  ]\th| \,dx \, dt  \les R^{n-1/2} \norm{\theta}_{E_T} \left( \int_0^T \|\th\|_{A_t}^2\,dt \right)^{1/2}.
\end{equation}
On the other hand, for a.e. $t \in (0,T)$, we have
\EQ{
\int (1-\phi_{2R })\La^{1/2}\th [ \La^{1/2},\phi_R  ]\th\,dx &= \int (1-\phi_{2R })\La^{1/2}\th\La^{1/2}(\th \phi_R)\,dx.
}
We need only estimates in the region $\{ x \not\in B_R \}$, where we have the pointwise bound
\begin{equation}
	|\Lambda^{1/2} (\theta \phi_R)(x)| \les \left\lvert \pv \int_{\R^n} \frac{ - \theta(y)\phi_R(y)}{|x-y|^{n+1/2}} \, dy \right\rvert \les |x|^{-(n+1/2)} \norm{\theta}_{L^1(B_R)} \les |x|^{-(n+1/2)} R^{n/2} \norm{\theta}_{L^2(B_R)}.
\end{equation}
This implies
\begin{equation}
	\label{eq:summyseries}
\begin{aligned}
	& \int |(1-\phi_{2R })\La^{1/2}\th\La^{1/2}(\th \phi_R)| \,dx\\
	&\quad \les  R^{n/2} \sum_{k=1}^{+\infty} \int_{B(2^{k} R) \setminus B(2^{k-1} R)} |x|^{-(n+1/2)} |\Lambda^{1/2} \theta(x,t)| \, dx \times \norm{\theta(\cdot,t)}_{L^2(B_R)} \\
	&\quad \les R^{n/2} \sum_{k=1}^{+\infty} (2^k R)^{n/2} (2^k R)^{-(n+1/2)} \norm{\Lambda^{1/2} \theta(\cdot,t)}_{L^2(B_{2^k R})} \times \norm{\theta(\cdot,t)}_{L^2(B_R)}.
	\end{aligned}
\end{equation}
To obtain the desired bound for the above term, we need to integrate in time at this point. Then, summing the geometric series in~\eqref{eq:summyseries}, we have
\begin{equation}
\begin{aligned}
	&\int_0^T \int |(1-\phi_{2R })\La^{1/2}\th\La^{1/2}(\th \phi_R)| \,dx \\
	&\quad \les R^{-1/2}  \sup_{k \in \N} \, \norm{\Lambda^{1/2} \theta}_{L^2_{t,x}(B_{2^k R} \times (0,T))} \norm{\theta(\cdot,t)}_{L^2_{t,x}(Q_{R,T})} \\
	&\quad\les R^{n-1/2} \norm{\theta}_{E_T} \left( \int_0^T \norm{\theta}_{A_t}^2 \, dt \right)^{1/2}.
	\end{aligned}
\end{equation}

Putting these estimates together and integrating in time, we see that
\begin{equation}
	\label{eq:intermediarystep}
\begin{aligned}
	&\sup_{t \in (0,T)} \int_{B_{R}} |\theta(x,t)|^2 \phi_R \, dx + 2\int_{0}^T \int_{B_{R/2}} |\Lambda^{1/2} \theta|^2 dx \, dt \leq \norm{\theta_0 \phi_R^{1/2}}^2_{L^2(B_{R})} \\
	&\quad+ CR^{n-1/2} \norm{\theta}_{E_T} \left( \int_0^T \norm{\theta}_{A_t}^2 \, dt \right)^{1/2} + \int_{0}^{T} \int_{\R^n} |\theta|^2 |v \cdot \nabla \phi_R| \, dx \, dt.
	\end{aligned}
\end{equation}

We now analyze the drift term: 
\begin{equation}
	\label{eq:drifttermest}
\begin{aligned}
	&\int_{0}^{T} \int_{\R^n} |\theta|^2 |v \cdot \nabla \phi_R| \, dx \, dt  \\
	&\quad \les R^{-1} \int_0^T \norm{v(\cdot,t)}_{L^p_x(B_R)} \norm{\theta(\cdot,t)}^2_{L^{2p'}_x(B_R)} \, dt \\
	&\quad \les R^{-1} \int_0^T \norm{v}_{L^\infty(0,t;L^p(B_R))} \norm{\theta(\cdot,t)}_{L^{\kappa}_x(B_R)}^{2\tau} \norm{\theta(\cdot,t)}_{L^2_{x}(B_R)}^{2(1-\tau)} \, dt \\
	&\quad \les R^{-1} \int_0^T \norm{v}_{L^\infty(0,t;L^p(B_R))} [\norm{\Lambda^{1/2} \theta(\cdot,t)}_{L^2(B_{2R})} + R^{-1/2} \norm{\theta(\cdot,t)}_{L^2(B_{2R})} ]^{2\tau} \norm{\theta(\cdot,t)}_{L^2(B_R)}^{2(1-\tau)} \, dt
	\end{aligned}
\end{equation}
 where $\tau = \tau(p) \in (0,1]$ satisfies
 \begin{equation}
 	\label{eq:taudef}
	\frac{\tau}{\kappa} + \frac{1-\tau}{2} = \frac{1}{2p'},
 \end{equation}
 and we have used Sobolev embedding $H^{1/2}(B_R) \into L^\kappa(B_R)$ as in~\eqref{eq:sobolevembedding} and~\eqref{eq:sobolevembeddingcontinued}.
 This is possible because $p \geq n$ and $n' = \kappa/2$ give $2p' \in [2,\kappa]$. Of course, $\tau = 1$ only when $p = n$. When $p > n$, we split the above product by Young's inequality:
 \begin{equation}
 \begin{aligned}
	&\int_{0}^{T} \int_{\R^n} |\theta|^2 |v \cdot \nabla \phi_R| \, dx \, dt \\
	&\quad \leq 2^{-2n-1} \norm{\Lambda^{1/2} \theta}_{L^2_{t,x}(Q_{2R,T})}^2  + \int_0^T \left[ R^{-1} +  C_p (R^{-1} \norm{v}_{L^\infty(0,t;L^p(B_R))})^{2/(1-\tau)} \right] \times \norm{\theta(\cdot,t)}_{L^2(B_{2R})}^2 \, dt\\
	&\quad \leq 2^{-n-1} |B_1|^2 R^{n} \norm{\theta}_{E_T}^2  +R^{n-1}\int_0^T \|\th\|_{A_t}^2\,dt+  C_p R^{n} \int_0^T \|v \|_{L^\infty(0,t;X_p)}^{2/(1-\tau)} \norm{\theta}_{A_t}^2 \, dt
	\end{aligned}
 \end{equation}
 since $R^{-1} \leq R^{-n/p} \leq 1$. When $p=n$, one does not split this term and instead uses smallness to close.
 We insert the above estimate into~\eqref{eq:intermediarystep}, divide by $|B_1|^2 (2R)^n$, and take the supremum over $R \geq 2$:
 \begin{equation}
\begin{aligned}
	&\norm{\theta}_{A_T}^2 + 2 \norm{\theta}_{E_T}^2 \leq C \norm{\theta_0}_{X_2}^2 + C R^{-1/2} \norm{\theta}_{E_T} \left( \int_0^T \norm{\theta}_{A_t}^2 \, dt \right)^{1/2} \\
	&\quad+ \frac{1}{2} \norm{\theta}_{E_T}^2 + C R^{-1}\int_0^T \|\th\|_{A_t}^2\,dt+  C_p \int_0^T \|v \|_{L^\infty(0,t;X_p)}^{2/(1-\tau)} \norm{\theta}_{A_t}^2 \, dt.
	\end{aligned}
\end{equation}
We split the product $\norm{\theta}_{E_T} \left( \int_0^T \norm{\theta}_{A_t} \, dt \right)^{1/2}$ with Young's inequality and absorb the $E_T$ terms on the LHS. This gives~\eqref{eq:energyestimatesuitableforGronwall}. 
Finally, taking $T = T_0 \ll_{p,V} 1$ gives~\eqref{eq:energyestimatelesssuitable}.
\end{proof}

\begin{remark}
\label{rmk:furtherobservations}
\begin{enumerate}
\item Observe that~\eqref{eq:energyestimatesuitableforGronwall} also implies a rate $\sim (T - T^*)^{-1}$ at which we lose control of $\norm{\theta}_{A_T}^2$ in the above estimates.

\item By substituting~\eqref{eq:energyestimatelesssuitable} back into the above computations, we may also show
	\begin{equation}
\label{eq:importantlowersemicontinuityfact}
	\int |\theta(x,t)|^2 \phi_R \, dx \leq \int |\theta_0|^2 \phi_R \, dx + C_p R^{n-1/2} (T^{1/2} + V T^{1-\tau}) \norm{\theta_0}_{X_2}^2.
\end{equation}
This requires adjusting the estimate~\eqref{eq:drifttermest}. We have
\begin{equation}
	\label{eq:drifttermestadjusted}
\begin{aligned}
	&\int_{0}^{T} \int_{\R^n} |\theta|^2 |v \cdot \nabla \phi_R| \, dx \, dt \\
	&\quad \les R^{-1} \int_0^T \norm{v}_{L^\infty(0,t;L^p(B_R))} [\norm{\Lambda^{1/2} \theta(\cdot,t)}_{L^2(B_{2R})} + R^{-1/2} \norm{\theta(\cdot,t)}_{L^2(B_{2R})} ]^{2\tau} \norm{\theta(\cdot,t)}_{L^2(B_R)}^{2(1-\tau)} \, dt \\
	&\quad \les V R^{n-1} \left[ \norm{\theta}_{E_T} + R^{-1/2} T^{1/2} \norm{\theta}_{A_T} \right]^{2\tau} \times (T^{1/2} \norm{\theta}_{A_T})^{2(1-\tau)} \\
	&\quad \les V C_p R^{n-1} T^{1-\tau} \norm{\theta_0}_{X_2}^2,
	\end{aligned}
\end{equation}
where we also use that $T_0 \leq 1$.
One may substitute this into~\eqref{eq:intermediarystep} to conclude~\eqref{eq:importantlowersemicontinuityfact}. 
This will be used in Lemma~\ref{lem:stabilityL2loc} to show that the energy inequality starting from the initial time holds.

\item We also see from~\eqref{eq:importantlowersemicontinuityfact} that if $R^{-n} \int_{B_R} |\theta_0|^2 \, dx$ decays as $R \to +\infty$ (for example, if $\theta_0$ belongs to the closure of test functions in $X_2$), then $\theta(\cdot,t)$ inherits an analogous decay, uniformly in $t \in (0,T_0)$, as $R \to +\infty$.

\end{enumerate}
\end{remark}

The following stability lemma will be used in the proof of Theorem \ref{thrm.existence}.
 
\begin{lemma}[Stability]\label{lem:stabilityL2loc}
Assume $\th_0^{(k)}\to \th_0$ in $L^2(K)$ for every compact set $K$   and $v^{(k)}\wstar v$ in  $L^\infty_t (X_p)_x(\R^n \times \R_+)$. Let $\th^{(k)}$ be the solution to~\eqref{eq:linear} with initial data $\th_0^{(k)}$ and drift $v^{(k)}$.
Assume that $\th^{(k)},\theta^{(k)}_0,v^{(k)}$ satisfy the assumptions of Lemma~\ref{lem:aprioriestimates}. Then, there exists a subsequence (still indexed by $k$) and $\theta$ solving~\eqref{eq:linear} with initial data $\th_0$ and drift $v$,
and 
\EQ{\label{all.your.limit.are.belong.to.us}
&\th^{(k)} \wstar \th \text{  in }L^\I(0,N;L^2(B_N))
\\&\th^{(k)}  \wto \th \text{   in } L^2(0,N ;H^{1/2}(B_N)),
\\& \La^{1/2} \th^{(k)}  \wto \La^{1/2}\th  \text{  in } L^2(0,N ;L^2(B_N)),
\\& \partial_t \th^{(k)}  \wto \partial_t \th \text{   in }L^2(0,N;H^{-(n+1)} (B_N)),
\\&\th^{(k)}  \to \th \text{  in } L^2(0,N ;L^2(B_N)),
} 
for every $N\in \N$.
Moreover, $\th$ satisfies the local energy inequalities \eqref{eq:localenergyeq} and \eqref{eq:localenergyineqtimeslice} and attains its initial data in the sense of \eqref{eq:attaininitialdata}. Hence, it satisfies the \emph{a priori} bounds in Lemma~\ref{lem:aprioriestimates}. Finally (when $n=2$), if $\theta^{(k)}$ is symmetric (resp. $\lambda$-DSS) for all $k \in \N$, then $\theta$ is symmetric (resp. $\lambda$-DSS). %
\end{lemma}

 
The proof of Lemma \ref{lem:stabilityL2loc} tailors the usual compactness methods for dissipative fluid models to~\eqref{eq:sqg}. 
Let
\[
X_0= H^{1/2}(B_N),\quad X = L^2(B_N),\quad X_1 = H^{-(n+1)}(B_N).
\]
It is well known that $H^{1/2}(B_N) \overset{\text{cpt}}{\into} L^2(B_N)$ \cite[Theorem 7.1]{HitchhikerGuideFractional}, 
 $L^2(B_N) \into H^{-(n+1)}(B_N)$, and all these spaces are reflexive. Therefore, the Aubin--Lions lemma~\cite{aubin} implies that
 \begin{equation}
 	\label{eq:aubinlions}
\left\lbrace  \th \in  L^2(0,T;X_0): \partial_t \th\in  L^2(0,T;X_1) \right\rbrace \overset{\text{cpt}}{\into} L^2(0,T;X)	
 \end{equation}
when the LHS is endowed with the obvious norm. 
When $\Omega \subset \R^n$ is a smooth, bounded domain, the norm of $H^{1/2}(\Om)$ satisfies
\[
\|\th\|^2_{H^{1/2}(\Om)} \sim \int_{\Om} |\th|^2\,dx +  \int_{\Om}\int_{\Om} \frac {|\th(x)-\th(y)|^2} {|x-y|^{n+1}}\,dx\,dy=: \|\th\|_{L^2(\Om)}^2 + [\th]_{H^{1/2}(\Om)}^2,
\] 
where $A \sim B$ means that $A \les B$ and $B \les A$. The term $[\theta]_{H^{1/2}(\Omega)}$ is the Gagliardo seminorm, and when $\Omega = \R^n$, it is equal to $\text{const.} \times \norm{ \Lambda^{1/2} \theta }_{L^2(\R^n)}$. See~\cite{HitchhikerGuideFractional} for further discussion. To apply~\eqref{eq:aubinlions}, one must relate $\La^{1/2} f$, restricted to a neighborhood of $\overline{\Omega}$, to the $1/2$-Sobolev norm of $f$. We also require an estimate on $\Lambda \theta$. We clarify this here before returning to the proof of Lemma~\ref{lem:stabilityL2loc}.

\begin{lemma}\label{lemma:fractionalsobolev}
Let $R \geq 1$ and $f \in X_2$ with $\Lambda^{1/2} f \in L^2(B_{2R})$. Then
\begin{equation}\label{eq:sobolevseminorm}
	[f]_{H^{1/2}(B_{3R/2})} \les R^{(n-1)/2} \norm{f}_{X_2} + \norm{\Lambda^{1/2} f}_{L^2(B_{2R})}
\end{equation}
 and
\begin{equation}\label{eq:sobolevinverse}
\| \La f \|_{H^{-1}(B_R)} \leq  R^{n/2}\|f\|_{X_2}. 
\end{equation}
\end{lemma}
 
\begin{proof}
Let $\phi_R\in C_0^\I(\R^n)$ with $\phi \equiv 1$ in $B_{3R/2}$ and vanishing off of $B_{2R}$. Then
\EQN{
 [f ]_{H^{1/2}(B_{3R/2})}^2\leq [f \phi_R ]_{H^{1/2}(B_{2R})}^2 \sim \| \La^{1/2} (\phi_{R} f )\|_{L^2(\R^2)}^2 &\lesssim  \|    \La^{1/2}  f\|_{L^2(B_{2R})}^2 + \| [\Lambda^{1/2},\phi_R] f \|_{L^2(\R^2)}^2\\&\lesssim  R^{(n-1)/2} \norm{f}_{X_2} + \norm{\Lambda^{1/2} f}_{L^2(B_{2R})},
}
where we apply the commutator estimate in Lemma~\ref{lem:muhcommutatorest}. This proves~\eqref{eq:sobolevseminorm}.

We now prove \eqref{eq:sobolevinverse}. Assume $g\in H_0^1(B_R)$. Then
\[
|\la \La f, g\ra_{H^{-1}(B_R),H^1_0(B_R)}| = |\la \La f , g\ra_{H^{-1}(\R^n),H^1_0(\R^n)}|\leq \bigg|\int_{B_{2R}} f \La g\,dx\bigg|+\bigg|\int_{B_{2R}^c} f \La g\,dx\bigg|:=I_1+I_2.
\]
For the first integral, since the domain of integration is bounded, we have
\begin{equation}
	\label{eq:I1est}
	|I_1|\lesssim R^{n/2} \| f \|_{X_2}  \|\La g\|_{L^2(\R^n)}   \sim  R^{n/2} \|f\|_{X_2} \|g\|_{H^1_0(B_N)}
\end{equation}
where we used the following computation to pass from the global to local norm:
\[
\int_{\R^2} |\La g|^2\,dx\sim   \int_{\R^2} (|\xi|^2 \widehat g) \widehat g\,d\xi \sim \int_\Om(-\Delta g) g \,dx =\int_\Om |\nb g|^2\,dx.
\]
On the other hand, for $I_2$, we have  
\begin{equation}
|I_2| \leq  \sum_{i=1}^\I \int_{(2R)^i\leq |x|\leq (2R)^{i+1}} \frac {f(x)} {|x|^{n+1}} \int g(y)\,dy \,dx,	
\end{equation}
where we used $|x|\leq 2|x-y|$ whenever $y\in B_R$ and $2R\leq |x|$. This leads to 
\begin{equation}
\label{eq:I2est}
|I_2| \lesssim  R^{n/2}  \|g\|_{L^2(B_R)} \|f \|_{X_2}.	
\end{equation}
Together,~\eqref{eq:I1est} and~\eqref{eq:I2est} yield the desired estimate.
\end{proof}

\begin{proof}[Proof of Lemma \ref{lem:stabilityL2loc}] We use freely the local energy estimates from Lemma~\ref{lem:aprioriestimates}, which gives the first three lines of~\eqref{all.your.limit.are.belong.to.us}.

\textit{1. Aubin--Lions estimates}. We momentarily omit the index $k$. All estimates will be uniform in $k$. First, Lemma~\ref{lemma:fractionalsobolev} gives that $\th\in L^2(0,T;H^{1/2} (B_N))$. Next, we show that $\partial_t \th = - \Lambda \theta - \div (v \theta) \in L^2_t H^{-(n+1)}_x(B_N)$ for any $N\in \N$. To estimate $\Lambda \theta$, we use Lemma \ref{lemma:fractionalsobolev}:
\[
\|\La \th \|_{H^{-1}(B_N)}\les N^{n/2} \| \th \|_{X_2}  \in L^2(0,T).
\]
Hence, $\La\th\in L^2(0,T; H^{-1}(B_N))\subset L^2(0,T;H^{-(n+1)}(B_N))$. Using H\"older's inequality and the Sobolev embedding $H^n(B_N) \into L^\infty(B_N)$, we have for any $\gamma\in L^2(\R_+;H^{n+1}_0(B_{N}))$ that 
\EQ{
\iint (v\th )\cdot \nb \ga \,dx\,dt &\lesssim T^{1/2} \|v\|_{A_T} \norm{\theta}_{A_T}  \|\ga\|_{L^2 (0,T;H^{n+1}(B_N))},
}
The embedding $X_p \into X_2$ and our assumptions then imply  $\div (v\th)\in L^2((0,T);H^{-(n+1)}(B_N))$. 
It follows that $\partial_t \th \in L^2((0,T);H^{-(n+1)}(B_N))$, and we may now apply the Aubin-Lions lemma to obtain the advertised convergence of a subsequence (still indexed by $k$) to a scalar $\th$.  
 

\textit{2. $\theta$ solves the PDE}. We verify that $\th$ solves \eqref{eq:linear}.  Fix $\phi\in C^\infty_0(\R^2\times \R_+)$. Let $B$ be a ball centered at the origin so that $\supp\phi \subset B$ for all $t$, and let $T$ be large enough that $\phi(x,t)\equiv 0$ for all $t>T$. Weak convergence immediately implies that $\partial_t \th^{(k)}\to \partial_t\th$ in the sense of distributions. For the advective term, convergence follows because $ \th^{(k)} \nb \phi \to \theta \nabla \phi \in L^2_{t,x}(B)$ and $v^{(k)} \wstar v$ in $L^\I(0,T;L^2(B))$ for all $T$.
For the diffusive term, we have 
 \EQN{
\bigg|\iint ( \th^{(k)}-\th )\Lambda\phi \,dx\,dt\bigg| &\lesssim_\phi \int_0^T\int  \frac {|\th^{(k)}-\th |} {(|x|+1)^{n+1}} \,dx\,dt.
\\&\lesssim \| \th^{(k)}-\th\|_{L^2(0,T;L^2(B_M))} +  \int_0^T \int_{|x|>M}  \frac {|\th^{(k)}-\th |} {(|x|+1)^{n+1}} \,dx\,dt.
}
The first term  vanishes as $k \to \I$, and the second term can be made small by taking $M$ large using the uniform-in-$k$ bounds in $A_T$.  

The statement about preservation of symmetry and scaling is obvious. 

\textit{3. Local energy inequality}. Finally, we establish the local energy inequality
\EQ{\label{ineq:lei2}	- \frac{1}{2} \iint |\theta|^2 \p_t \phi \, dx \,dt + \la \Lambda \theta, \theta \phi \ra \leq \frac 1 2 \int |\th_0|^2 \phi \,dx +\frac{1}{2} \iint |\theta|^2  v \cdot \nabla \phi \, dx \,dt,}
for every non-negative function $\phi\in C_0^\I(\R^2\times [0,\I))$.  This is inherited from the corresponding inequalities for the approximating sequence. The above form of the local energy inequality is justified for the approximating sequence because the data is attained in the sense of~\eqref{eq:attaininitialdata}.
Most of the work is in analyzing the diffusive term. Note that, as in Lemma~\ref{lem:aprioriestimates}, 
\[
\la \Lambda \theta, \theta \phi \ra = \la \La^{1/2} \th , (\La^{1/2}\th) \phi \ra + \la \La^{1/2}\th, [\La^{1/2},\phi]\th\ra. 
\]
By weak convergence, we have 
\[
 \la \La^{1/2} \th , (\La^{1/2}\th) \phi \ra \leq \limsup_{k\to \I} \, \la \La^{1/2} \th^{(k)} , (\La^{1/2}\th^{(k)}) \phi \ra, 
\]
implying
\EQ{ \la \La^{1/2} \th , (\La^{1/2}\th) \phi \ra &\leq \limsup_{k\to \I } \big( \la \La \th^{(k)},\th^{(k)}\phi \ra - \la \La^{1/2}\th^{(k)}, [\La^{1/2},\phi]\th^{(k)}\ra\big)
\\&\leq \limsup_{k\to \I } \big( \frac{1}{2} \iint |\theta^{(k)}|^2  v^{(k)} \cdot \nabla \phi \, dx \,dt  + \frac{1}{2} \iint |\theta^{(k)}|^2 \p_t \phi \, dx \,dt  
\\&\quad +\frac 1 2 \int |\th_0^{(k)}|^2\phi \,dx- \la \La^{1/2}\th^{(k)}, [\La^{1/2},\phi]\th^{(k)}\ra \big).
}
By strong convergence the term with the time derivative converges to the corresponding term for $\th$. The same is true for the advective term, where we also use the convergence $v^{(k)} \wstar v$ in $(L^\infty_t L^p_x)_\loc(\R^n \times [0,+\infty))$, $p > n$, and $\theta^{(k)} \to \theta$ in $(L^2_t L^q_x)_\loc(\R^n \times [0,+\infty))$, $q < \kappa$. The initial terms converge by assumption. 
We therefore need to show
\[
\int_0^T\int \La^{1/2}\th^{(k)} [\La^{1/2},\phi]\th^{(k)} \,dx\,dt \to \int_0^T\int \La^{1/2}\th [\La^{1/2},\phi]\th\,dx\,dt.
\]
 Let $\bar \th = \th-\th^{(k)}$. For the claim about commutators we need to show
\[
\int_0^T\int \La^{1/2}\bar \th [\La^{1/2},\phi]\th \,dx\,dt+\int_0^T\int \La^{1/2}\th^{(k)}[\La^{1/2},\phi]\bar \th \,dx\,dt \to 0.
\]
Let $N\in \N$ be large enough so that $\supp\phi \subset B_N$. The terms above are the sum of the following four integrals:
\EQ{
&I_1 = \int_0^T\int \La^{1/2}\bar \th \phi_{2N} [\La^{1/2},\phi]\th \,dx\,dt,
\\&I_2=\int_0^T\int \La^{1/2}\th^{(k)} \phi_{2N} [\La^{1/2},\phi]\bar \th \,dx\,dt,
\\&J_1=\int_0^T\int \La^{1/2}\bar \th (1-\phi_{2N}) [\La^{1/2},\phi]\th \,dx\,dt,
\\&J_2=\int_0^T\int \La^{1/2}\th^{(k)} (1-\phi_{2N}) [\La^{1/2},\phi]\bar \th \,dx\,dt.
}
We show each integral above vanishes in the limit.

For $I_1$, note that $\phi_{2N} [\La^{1/2},\phi]\th\in L^2(0,T;L^2( B_{2N}))$. So, by weak convergence of $\La^{1/2}\th^{(k)}$ to $\La^{1/2}\th$ in $L^2(0,T;L^2( B_{2N}))$, we have $I_1\to 0$ as $n\to \I$.

For $I_2$, we have 
\[
I_2\lesssim_{N} \|\th\|_{E_T}\|    [\La^{1/2},\phi]\bar \th \|_{L^2(B_{2N}\times [0,T])}.
\]
This vanishes as $k\to \I$ by \eqref{eq:commutatorest} and strong convergence of $\th^{(k)}\to \th$ in $L^2(B_{4N}\times [0,T])$.

For $J_1$ we note that 
\begin{equation}
	|\Lambda^{1/2} (\theta \phi)(x)| \les  \bigg|  \int_{\R^n} \frac{ - \theta(y)\phi(y)}{|x-y|^{n+1/2}} \, dy \bigg| \les |x|^{-(n+1/2)} \norm{\theta}_{L^1(B_N)} \les |x|^{-(n+1/2)} N^{n/2} \norm{\theta}_{L^2(B_N)}.
\end{equation}
Then, by support considerations,
\EQ{
J_1 &= \int_0^T\int \La^{1/2}\bar \th (1-\phi_{2N})  \La^{1/2}(\phi \th)   \,dx\,dt
\\&=-\int_0^T   \int_{|x|\geq 2N} \La^{1/2}\bar \th(x) \int \frac {  \th(y)\phi(y)} {|x-y|^{n+1/2}}  \,dy\,dx \,dt
=:-J_1^1-J_1^2,
}
where 
\[
J_1^1=\int_0^T   \int_{2N\leq |x| < N_*} \La^{1/2}\bar \th(x) \int \frac {  \th(y)\phi(y)} {|x-y|^{n+1/2}}  \,dy\,dx \,dt,
\]
and
\[
J_1^2=\int_0^T   \int_{N_*\leq |x|  } \La^{1/2}\bar \th(x) \int \frac {  \th(y)\phi(y)} {|x-y|^{n+1/2}}  \,dy\,dx \,dt.
\]
For any $N_*$, $J_1^1\to 0$ by weak convergence  of $\La^{1/2}\th^{(k)}$ to $\La^{1/2}\th$ in $L^2(0,T;L^2( B_{2N}))$. By taking $N_*$ sufficiently  large, we can make $J_1^2$ is arbitrarily small.

For $J_2$ we have
\EQ{
J_2&\leq \int_0^T\int \La^{1/2}\th^{(k)} (1-\phi_{2N})  \frac {N^{n/2}} {|x|^{n+1/2}} \|\bar \th\|_{L^2(B_{N})} \,dx\,dt
\\& \lesssim_{N,\phi} \| \bar \th\|_{L^2(B_N\times [0,T]) }\|\th\|_{E_T},
}
which vanishes as $N \to \I$.

Convergence to the initial data in the sense \eqref{eq:attaininitialdata} now follows from standard arguments (cf.~\cite{KikuchiSeregin,KwonTsai} in the Navier-Stokes setting). One way is to use
\begin{equation}
	\int |\theta(x,t)|^2 \phi_R \, dx \leq \int |\theta_0|^2 \phi_R \, dx + o_{t \to 0^+}(1),
\end{equation}
which follows from taking $k \to +\infty$ in the estimate~\eqref{eq:importantlowersemicontinuityfact} for $\theta^{(k)}$. Here, the strong convergence $\theta^{(k)}_0 \to \theta_0$ in $L^2_\loc(\R^n)$ plays a crucial role. 
\end{proof}

\begin{corollary}
\label{cor:linftysolsattainstrongly}
The bounded solutions from Lemma~\ref{lem:Linftytheory} attain their initial data strongly in $L^p_\loc(\R^n)$ for all $p < +\infty$.
\end{corollary}
\begin{proof}
The $p=2$ case follows from Lemma~\ref{lem:stabilityL2loc} (Stability) and the approximation procedure in the existence part of Lemma~\ref{lem:Linftytheory}. The $p > 2$ case follows from interpolation.
\end{proof}

We now consider solutions with initial data belonging to $X_q$, where $q > 2$. These results are corollaries of Lemmas~\ref{lem:aprioriestimates} and~\ref{lem:stabilityL2loc}.

\begin{corollary}[\emph{A priori} estimates in $X_q$]\label{cor:highernorms}
Let $\theta$ be the unique global solution from~Lemma~\ref{lem:Linftytheory}. Let $q > 2$.  
Then $|\theta|^{q/2}$ satisfies the \emph{a priori} estimates in Lemma~\ref{lem:aprioriestimates}.
\end{corollary} 

\begin{proof}

We multiply the equation by $q |\theta|^{q-2} \theta$. This gives
\begin{equation}
\label{eq:insertedinto}
	\p_t |\theta|^q + q |\theta|^{q-2} \theta \Lambda \theta + v \cdot \nabla |\theta|^q = 0.
\end{equation}
We recall the pointwise inequality (see Proposition~3.3 and the proof of Lemma~3.3 in Ju's paper~\cite{jumaximumprinciple})
\begin{equation}
	\label{insertmeback}
	q |\theta|^{q-2} \theta \Lambda \theta \geq 2 |\theta|^{q/2} \Lambda (|\theta|^{q/2}).
\end{equation}
Inserting~\eqref{insertmeback} into~\eqref{eq:insertedinto}, we have
\begin{equation}
	\label{lplocalenergyineq}
	\p_t |\theta|^q + 2 |\theta|^{q/2} \Lambda (|\theta|^{q/2}) + v \cdot \nabla |\theta|^q \leq 0.
\end{equation}
In other words, $|\theta|^{q/2}$ satisfies the local energy inequality~\eqref{eq:linearinequality} in~Lemma~\ref{lem:aprioriestimates} with $|\theta|^{q/2}$ replacing $\theta$. Finally, Lemma~\ref{lem:aprioriestimates} gives the required estimates.\footnote{Technically, under the regularity assumptions of Lemma~\ref{lem:Linftytheory}, it is not obvious that $|\theta|^{q/2} \in E_T$, which is needed to apply Lemma~\ref{lem:aprioriestimates}. Rather, one may apply the estimates to the time-translated function $|\theta(x,t+t_0)|^{q/2}$, which belongs to $E_T$ for all $t_0 > 0$ under the given assumptions, and then take $t_0 \to 0^+$.}
\end{proof}
 
We can specialize these bounds from the linear equation to~\eqref{eq:sqg}.

\begin{corollary}[\emph{A priori} estimates in $X_q$ for SQG]\label{lemma.bounds.nonlinear.Xp}
Let $p > 2$ and $\th$ be a solution of~\eqref{eq:sqg} satisfying the properties in the existence part of Theorem~\ref{thrm.existence}. 
Then there exists $T_*=T_*(\|\th_0\|_{X_p},p) > 0$, \emph{not depending on $\|\th_0\|_{L^\I(\R^2)}$}, so that
\EQ{
\norm{|\theta|^{p/2}}_{A_{T_*}}^2 + 	\norm{|\theta|^{p/2}}_{E_{T_*}}^2 \leq 2  \norm{|\theta_0|^{p/2}}^2_{X_2} =2 \|\th_0\|_{X_p}^p,
}
\EQ{
	\label{x2estsigh}
\norm{\theta}_{A_{T_*}}^2 + 	\norm{\theta}_{E_{T_*}}^2 \leq 2  \norm{\theta_0}^2_{X_2}.
}
Rescaling gives that, for all $T \geq T_*$,
\EQ{
\norm{|\theta|^{p/2}}_{A_{T}^{T/T_*}}^2 + 	\norm{|\theta|^{p/2}}_{E_{T}^{T/T_*}}^2 \leq 2 \|\th_0\|_{X_p}^p,
}
\EQ{
\norm{\theta}_{A_{T}^{T/T_*}}^2 + 	\norm{\theta}_{E_{T}^{T/T_*}}^2 \leq 2 \|\th_0\|_{X_2}^2.
}
\end{corollary}

Again, we emphasize that the above bounds are independent of the $L^\infty$ norm.


\begin{proof}
By Lemma \ref{lem:stabilityL2loc}, we have
\begin{equation}
\begin{aligned}
	\norm{|\theta|^{p/2}}_{A_T}^2  &\les \norm{|\theta_0|^{p/2}}^2_{X_2}  +\int_0^T \||\theta|^{p/2}\|_{A_t}^2\,dt+  C_p \int_0^T \|\vec{R}^\perp \th \|_{L^\I (0,t;X_p)}^{2/(1-\tau)} \||\theta|^{p/2}\|_{A_t}^2 \,dt \\
		 &\les \norm{|\theta_0|^{p/2}}^2_{X_2}  +\int_0^T \||\theta|^{p/2}\|_{A_t}^2\,dt+  C_p  \int_0^T \|\th\|_{L^\I (0,t;X_p)}^{2/(1-\tau)} \||\theta|^{p/2}\|_{A_t}^2 \,dt \\
	 &\les  \norm{|\theta_0|^{p/2}}^2_{X_2}   +\int_0^T \||\theta|^{p/2}\|_{A_t}^2\,dt+  C_p  \int_0^T \||\th|^{p/2}\|_{A_t}^{2+4/(p(1-\tau))}   \,dt,
	\end{aligned}
\end{equation}
where we used boundedness of the Riesz transforms on $X_p$ in the symmetric case. We now use a standard Gr\"onwall inequality~\cite{BT8} to obtain a time $T_*=T_*( |\theta_0|^{p/2} , p)$ so that the advertised bound holds. The estimate~\eqref{x2estsigh} then follows from Lemma~\ref{lem:aprioriestimates}.
\end{proof}

Finally, we present an analogue of Lemma~\ref{lem:stabilityL2loc} (Stability) at the level of $|\theta|^{q/2}$, and which may be regarded as a corollary of the proof of Lemma~\ref{lem:stabilityL2loc}.

\begin{corollary}[$X_q$ stability]
\label{cor:lpstability}
Let the assumptions of Lemma~\ref{lem:stabilityL2loc} be satisfied. Let $q > 2$. Further assume that $\theta_0^{(k)} \to \theta_0$ in $L^q_\loc(\R^n)$, $|\theta|^{q/2} \in A_T \cap E_T$ for every $T > 0$, the initial data is attained in the sense that $\theta^{(k)}(\cdot,t) \to \theta_0^{(k)}$ in $L^q_\loc(\R^n)$ as $t \to 0^+$, and $|\theta^{(k)}|^{q/2}$ satisfies the local energy inequality~\eqref{lplocalenergyineq} in the sense of distributions. Then $|\theta|^{q/2} \in A_T \cap E_T$ for every $T>0$, $\theta(\cdot,t) \to \theta_0$ in $L^q_\loc(\R^n)$ as $t \to 0^+$, and $|\theta|^{q/2}$ satisfies the local energy inequality~\eqref{lplocalenergyineq} in the sense of distributions. Hence, $|\theta|^{q/2}$ satisfies the \emph{a priori} estimates in Lemma~\ref{lem:aprioriestimates}.
\end{corollary}
\begin{proof}
The stability arguments in the proof of Lemma~\ref{lem:stabilityL2loc} hold at the level of $|\theta|^{q/2}$ with minor adjustments: (1) One does not estimate the time derivative to apply the Aubin-Lions lemma, since we are exploiting an inequality for $|\theta|^q$ rather than an equality.  Rather, one uses the known convergence $\theta^{(k)} \to \theta$ in $L^2_\loc(\R^n \times \R_+)$ and interpolates with the \emph{a priori} estimates of $\theta$ and $|\theta|^{q/2}$ in $A_T$ and $E_T$ as necessary. This allows to skip Step 1. (Step 2 is also skipped because there is no equation to verify for $|\theta|^{q/2}$).) (2) For convergence to the initial data, we recall that $L^q(K)$, $K \subset \R^n$ compact, is a uniformly convex Banach space. Hence, when (i) $\norm{\theta(\cdot,t)}_{L^q(K)} \to \norm{\theta_0}_{L^q(K)}$ and (ii) $\theta(\cdot,t) \wto \theta_0$ in $L^q(K)$ as $t \to 0^+$, then $\theta(\cdot,t) \to \theta_0$ in $L^q(K)$ as $t \to 0^+$. This appears in the following way: (i) follows from the arguments in the last paragraph of Lemma~\ref{lem:stabilityL2loc}, and (ii) follows from $\theta \in C_{\rm wk}([0,T];L^2(K))$ and a density argument in $L^\infty_t L^q_x (K \times (0,T))$.
\end{proof}

\section{Solutions with unbounded data}
\label{sec:unboundedsolutions}

In this section, we prove Theorem~\ref{thrm.existence}.

\begin{proof}[Proof of Theorem~\ref{thrm.existence}] 
We approximate the initial data by bounded truncations that preserve the symmetries of the original data. 
Let $\th_0\in X_p$ be given and let $\th_0^{(m)}$ be the truncation of $\th_0$, i.e.~$\th_0^{(m)}=m \operatorname{sgn} \th_0$ if $|\th_0|>m$ and $\th_0^{(m)}=\th_0$ otherwise. Clearly, $\theta^{(m)} \to \theta_0$ in $L^p_\loc(\R^2)$, and the truncation does not affect symmetries or self-similarity. By Theorem~\ref{thm:boundedsolutions}, we have global existence of a bounded, smooth, symmetric solution $\th^{(m)}$ for each $\th_0^{(m)}$. By Corollary~\ref{cor:highernorms} and Corollary~\ref{lemma.bounds.nonlinear.Xp}, 
we have uniform-in-$m$ bounds on $\th^{(m)}$ and $|\th^{(m)}|^{p/2}$ in $A_{T}^{T/T_*}\cap E_T^{T/T_*}$, where $T_*$ can be taken to depend only on $\|\th_0\|_{X_p}$ and $p$. By increasing $T \geq T_*$, these estimates extend to arbitrarily large times. Hence, we may apply the linear stability results of Lemma~\ref{lem:stabilityL2loc} and Corollary~\ref{cor:lpstability} to pass to the limit. We must further justify that $v$, defined to be the weak-$\ast$ limit of $v^{(m)} := \vec{R}^\perp \theta^{(m)}$ in $L^\infty_t (X_p)_x(\R^2 \times (0,T))$ for all $T > 0$, satisfies $v = \vec{R}^\perp \theta$. This follows from the convergence $\theta^{(m)} \to \theta$ in $L^2_\loc(\R^n \times [0,+\infty))$ and the decay properties of the Riesz transform in Remark~\ref{rmk:riesztransformsandconvergence}. The existence proof is complete.  Finally, the claimed estimates were the subject of Corollary~\ref{lemma.bounds.nonlinear.Xp}.  
\end{proof}

\subsubsection*{Acknowledgments}
We sincerely thank the three referees for their valuable work reviewing our manuscript. DA was supported by the NDSEG Fellowship and NSF Postdoctoral Fellowship Grant No. 2002023. ZB was supported in part by the Simons Foundation Grant No. 635438.

\bibliography{bibliography}

\providecommand{\bysame}{\leavevmode\hbox to3em{\hrulefill}\thinspace}
\providecommand{\MR}{\relax\ifhmode\unskip\space\fi MR }
\providecommand{\MRhref}[2]{%
  \href{http://www.ams.org/mathscinet-getitem?mr=#1}{#2}
}
\providecommand{\href}[2]{#2}
\begin{thebibliography}{10}

\bibitem{AbidiHmidi}
Hammadi Abidi and Taoufik Hmidi, \emph{On the global well-posedness of the
  critical quasi-geostrophic equation}, SIAM J. Math. Anal. \textbf{40} (2008),
  no.~1, 167--185. \MR{2403316}

\bibitem{globalweakbesov}
Dallas Albritton and Tobias Barker, \emph{Global weak {B}esov solutions of the
  {N}avier-{S}tokes equations and applications}, Arch. Ration. Mech. Anal.
  \textbf{232} (2019), no.~1, 197--263. \MR{3916974}

\bibitem{albrittonbeekie2020long}
Dallas Albritton and Rajendra Beekie, \emph{Long-time behavior of scalar
  conservation laws with critical dissipation}, arXiv preprint arXiv:2010.09065
  (2020).

\bibitem{aubin}
Jean-Pierre Aubin, \emph{Un th\'eor\`eme de compacit\'e}, C. R. Acad. Sci.
  Paris \textbf{256} (1963), 5042--5044. \MR{0152860}

\bibitem{bahourichemindanchin}
Hajer Bahouri, Jean-Yves Chemin, and Rapha\"el Danchin, \emph{Fourier analysis
  and nonlinear partial differential equations}, Grundlehren der Mathematischen
  Wissenschaften [Fundamental Principles of Mathematical Sciences], vol. 343,
  Springer, Heidelberg, 2011. \MR{2768550}

\bibitem{basson}
Arnaud Basson, \emph{Solutions spatialement homog\'enes adapt\'ees au sens de
  {C}affarelli, {K}ohn et {N}irenberg des \'equations de {N}avier-{S}tokes},
  2006, {T}h\`ese--Universit\'e d'\'Evry.

\bibitem{BK1}
Zachary Bradshaw and Igor Kukavica, \emph{Existence of suitable weak solutions
  to the {N}avier-{S}tokes equations for intermittent data}, J. Math. Fluid
  Mech. \textbf{22} (2020), no.~1, Paper No. 3, 20. \MR{4040641}

\bibitem{BKT1}
Zachary Bradshaw, Igor Kukavica, and Tai-Peng Tsai, \emph{Existence of global
  weak solutions to the {N}avier-{S}tokes equations in weighted spaces}, arXiv
  preprint arXiv:1910.06929 (2019).

\bibitem{bradshawtsaiII}
Zachary Bradshaw and Tai-Peng Tsai, \emph{Forward discretely self-similar
  solutions of the {N}avier-{S}tokes equations {II}}, Ann. Henri Poincar\'e
  \textbf{18} (2017), no.~3, 1095--1119. \MR{3611025}

\bibitem{bradshawtsairot}
\bysame, \emph{Rotationally corrected scaling invariant solutions to the
  {N}avier--{S}tokes equations}, Comm. Partial Differential Equations
  \textbf{42} (2017), no.~7, 1065--1087. \MR{3691390}

\bibitem{bradshawtsaibesov}
\bysame, \emph{Discretely self-similar solutions to the {N}avier-{S}tokes
  equations with {B}esov space data}, Arch. Ration. Mech. Anal. \textbf{229}
  (2018), no.~1, 53--77. \MR{3799090}

\bibitem{BT8}
\bysame, \emph{Global existence, regularity, and uniqueness of infinite energy
  solutions to the {N}avier-{S}tokes equations}, Comm. Partial Differential
  Equations \textbf{45} (2020), no.~9, 1168--1201. \MR{4134389}

\bibitem{BuckmasterShkollerVlad}
Tristan Buckmaster, Steve Shkoller, and Vlad Vicol, \emph{Nonuniqueness of weak
  solutions to the {SQG} equation}, Comm. Pure Appl. Math. \textbf{72} (2019),
  no.~9, 1809--1874. \MR{3987721}

\bibitem{BuckmasterVicolAnnals}
Tristan Buckmaster and Vlad Vicol, \emph{Nonuniqueness of weak solutions to the
  {N}avier-{S}tokes equation}, Ann. of Math. (2) \textbf{189} (2019), no.~1,
  101--144. \MR{3898708}

\bibitem{CV}
Luis~A. Caffarelli and Alexis Vasseur, \emph{Drift diffusion equations with
  fractional diffusion and the quasi-geostrophic equation}, Ann. of Math. (2)
  \textbf{171} (2010), no.~3, 1903--1930. \MR{2680400}

\bibitem{calderon90}
Calixto~P. Calder\'on, \emph{Existence of weak solutions for the
  {N}avier-{S}tokes equations with initial data in {$L^p$}}, Trans. Amer. Math.
  Soc. \textbf{318} (1990), no.~1, 179--200. \MR{968416}

\bibitem{ConstantinCordobaGancedoWu}
Dongho Chae, Peter Constantin, Diego C\'{o}rdoba, Francisco Gancedo, and
  Jiahong Wu, \emph{Generalized surface quasi-geostrophic equations with
  singular velocities}, Comm. Pure Appl. Math. \textbf{65} (2012), no.~8,
  1037--1066. \MR{2928091}

\bibitem{wolfchael2loc}
Dongho Chae and J\"{o}rg Wolf, \emph{Existence of discretely self-similar
  solutions to the {N}avier-{S}tokes equations for initial value in
  {$L_{loc}^2(\Bbb R^3)$}}, Ann. Inst. H. Poincar\'{e} Anal. Non Lin\'{e}aire
  \textbf{35} (2018), no.~4, 1019--1039. \MR{3795025}

\bibitem{chengkwon2020non}
Xinyu Cheng, Hyunju Kwon, and Dong Li, \emph{Non-uniqueness of steady-state
  weak solutions to the surface quasi-geostrophic equations}, arXiv preprint
  arXiv:2007.09591 (2020).

\bibitem{coifmanrochbergweiss}
R.~R. Coifman, R.~Rochberg, and Guido Weiss, \emph{Factorization theorems for
  {H}ardy spaces in several variables}, Ann. of Math. (2) \textbf{103} (1976),
  no.~3, 611--635. \MR{412721}

\bibitem{ConstantinCordobaWu}
Peter Constantin, Diego Cordoba, and Jiahong Wu, \emph{On the critical
  dissipative quasi-geostrophic equation}, vol.~50, 2001, Dedicated to
  Professors Ciprian Foias and Roger Temam (Bloomington, IN, 2000),
  pp.~97--107. \MR{1855665}

\bibitem{ConstantinETiti}
Peter Constantin, Weinan E, and Edriss~S. Titi, \emph{Onsager's conjecture on
  the energy conservation for solutions of {E}uler's equation}, Comm. Math.
  Phys. \textbf{165} (1994), no.~1, 207--209. \MR{1298949}

\bibitem{ConstantinMajdaTabak}
Peter Constantin, Andrew~J. Majda, and Esteban Tabak, \emph{Formation of strong
  fronts in the {$2$}-{D} quasigeostrophic thermal active scalar}, Nonlinearity
  \textbf{7} (1994), no.~6, 1495--1533. \MR{1304437}

\bibitem{ConstTarfVicolCMP}
Peter Constantin, Andrei Tarfulea, and Vlad Vicol, \emph{Long time dynamics of
  forced critical {SQG}}, Comm. Math. Phys. \textbf{335} (2015), no.~1,
  93--141. \MR{3314501}

\bibitem{ConstVicolGAFA}
Peter Constantin and Vlad Vicol, \emph{Nonlinear maximum principles for
  dissipative linear nonlocal operators and applications}, Geom. Funct. Anal.
  \textbf{22} (2012), no.~5, 1289--1321. \MR{2989434}

\bibitem{HitchhikerGuideFractional}
Eleonora Di~Nezza, Giampiero Palatucci, and Enrico Valdinoci,
  \emph{Hitchhiker's guide to the fractional {S}obolev spaces}, Bull. Sci.
  Math. \textbf{136} (2012), no.~5, 521--573. \MR{2944369}

\bibitem{Duoand}
Javier Duoandikoetxea, \emph{Fourier analysis}, Graduate Studies in
  Mathematics, vol.~29, American Mathematical Society, Providence, RI, 2001,
  Translated and revised from the 1995 Spanish original by David Cruz-Uribe.
  \MR{1800316}

\bibitem{elgindi2019singular}
Tarek~M. Elgindi and In-Jee Jeong, \emph{On singular vortex patches, i:
  Well-posedness issues}, arXiv preprint arXiv:1903.00833 (2019).

\bibitem{elgindijeong}
\bysame, \emph{On singular vortex patches, {II}: {L}ong-time dynamics}, Trans.
  Amer. Math. Soc. \textbf{373} (2020), no.~9, 6757--6775. \MR{4155190}

\bibitem{elgindijoungsymmetries}
\bysame, \emph{Symmetries and critical phenomena in fluids}, Comm. Pure Appl.
  Math. \textbf{73} (2020), no.~2, 257--316. \MR{4054357}

\bibitem{FDLR1}
Pedro~Gabriel Fern\'{a}ndez-Dalgo and Pierre~Gilles Lemari\'{e}-Rieusset,
  \emph{Weak solutions for {N}avier-{S}tokes equations with initial data in
  weighted {$L^2$} spaces}, Arch. Ration. Mech. Anal. \textbf{237} (2020),
  no.~1, 347--382. \MR{4090469}

\bibitem{GallayWayne}
Thierry Gallay and C.~Eugene Wayne, \emph{Global stability of vortex solutions
  of the two-dimensional {N}avier-{S}tokes equation}, Comm. Math. Phys.
  \textbf{255} (2005), no.~1, 97--129. \MR{2123378}

\bibitem{guillodsverak}
Julien {Guillod} and Vladim{\'i}r {{\v S}ver{\'a}k}, \emph{{Numerical
  investigations of non-uniqueness for the Navier-Stokes initial value problem
  in borderline spaces}}, ArXiv e-prints (2017).

\bibitem{GustafsonKangTsai}
Stephen Gustafson, Kyungkeun Kang, and Tai-Peng Tsai, \emph{Interior regularity
  criteria for suitable weak solutions of the {N}avier-{S}tokes equations},
  Comm. Math. Phys. \textbf{273} (2007), no.~1, 161--176. \MR{2308753}

\bibitem{isett2020direct}
Philip Isett and Andrew Ma, \emph{A direct approach to nonuniqueness and
  failure of compactness for the sqg equation}, arXiv preprint arXiv:2007.03078
  (2020).

\bibitem{IsettVicol}
Philip Isett and Vlad Vicol, \emph{H\"{o}lder continuous solutions of active
  scalar equations}, Ann. PDE \textbf{1} (2015), no.~1, Art. 2, 77.
  \MR{3479065}

\bibitem{jiasverakselfsim}
Hao Jia and Vladim{\'i}r {\v S}ver{\'a}k, \emph{Local-in-space estimates near
  initial time for weak solutions of the {N}avier-{S}tokes equations and
  forward self-similar solutions}, Invent. Math. \textbf{196} (2014), no.~1,
  233--265. \MR{3179576}

\bibitem{jiasverakillposed}
\bysame, \emph{Are the incompressible 3d {N}avier-{S}tokes equations locally
  ill-posed in the natural energy space?}, J. Funct. Anal. \textbf{268} (2015),
  no.~12, 3734--3766. \MR{3341963}

\bibitem{jumaximumprinciple}
Ning Ju, \emph{The maximum principle and the global attractor for the
  dissipative 2{D} quasi-geostrophic equations}, Comm. Math. Phys. \textbf{255}
  (2005), no.~1, 161--181. \MR{2123380}

\bibitem{KikuchiSeregin}
Norio Kikuchi and Gregory Seregin, \emph{Weak solutions to the {C}auchy problem
  for the {N}avier-{S}tokes equations satisfying the local energy inequality},
  Nonlinear equations and spectral theory, Amer. Math. Soc. Transl. Ser. 2,
  vol. 220, Amer. Math. Soc., Providence, RI, 2007, pp.~141--164. \MR{2343610}

\bibitem{VariationsThemeCaffVass2009}
Aexander Kiselev and Fedor Nazarov, \emph{A variation on a theme of
  {C}affarelli and {V}asseur}, Zap. Nauchn. Sem. S.-Peterburg. Otdel. Mat.
  Inst. Steklov. (POMI) \textbf{370} (2009), no.~Kraevye Zadachi
  Matematichesko\u{\i} Fiziki i Smezhnye Voprosy Teorii Funktsi\u{\i}. 40,
  58--72, 220. \MR{2749211}

\bibitem{KiselevNazarovVolbergInventiones2007}
Alexander Kiselev, Fedor Nazarov, and Alexander Volberg, \emph{Global
  well-posedness for the critical 2{D} dissipative quasi-geostrophic equation},
  Invent. Math. \textbf{167} (2007), no.~3, 445--453. \MR{2276260}

\bibitem{korobkovtsai}
Mikhail Korobkov and Tai-Peng Tsai, \emph{Forward self-similar solutions of the
  {N}avier-{S}tokes equations in the half space}, Anal. PDE \textbf{9} (2016),
  no.~8, 1811--1827. \MR{3599519}

\bibitem{KwonTsai}
Hyunju Kwon and Tai-Peng Tsai, \emph{Global {N}avier-{S}tokes flows for
  non-decaying initial data with slowly decaying oscillation}, Comm. Math.
  Phys. \textbf{375} (2020), no.~3, 1665--1715. \MR{4091509}

\bibitem{lapeyre2017surface}
Guillaume Lapeyre, \emph{Surface quasi-geostrophy}, Fluids \textbf{2} (2017),
  no.~1, 7.

\bibitem{Lazar1}
Omar Lazar, \emph{Global existence for the critical dissipative surface
  quasi-geostrophic equation}, Comm. Math. Phys. \textbf{322} (2013), no.~1,
  73--93. \MR{3073158}

\bibitem{Lazar2}
\bysame, \emph{Global and local existence for the dissipative critical {SQG}
  equation with small oscillations}, J. Math. Fluid Mech. \textbf{17} (2015),
  no.~3, 533--549. \MR{3383927}

\bibitem{lemarie2002}
Pierre~Gilles Lemari\'e-Rieusset, \emph{Recent developments in the
  {N}avier-{S}tokes problem}, Chapman \& Hall/CRC Research Notes in
  Mathematics, vol. 431, Chapman \& Hall/CRC, Boca Raton, FL, 2002.
  \MR{1938147}

\bibitem{lemarie2016}
\bysame, \emph{The {N}avier-{S}tokes problem in the 21st century}, CRC Press,
  Boca Raton, FL, 2016. \MR{3469428}

\bibitem{leray}
Jean Leray, \emph{Sur le mouvement d'un liquide visqueux emplissant l'espace},
  Acta Math. \textbf{63} (1934), no.~1, 193--248. \MR{1555394}

\bibitem{DongLiFractional}
Dong Li, \emph{On {K}ato-{P}once and fractional {L}eibniz}, Rev. Mat. Iberoam.
  \textbf{35} (2019), no.~1, 23--100. \MR{3914540}

\bibitem{MaekawaMiura-drift}
Yasunori Maekawa and Hideyuki Miura, \emph{On fundamental solutions for
  non-local parabolic equations with divergence free drift}, Adv. Math.
  \textbf{247} (2013), 123--191. \MR{3096796}

\bibitem{MMPEnergy}
Yasunori Maekawa, Hideyuki Miura, and Christophe Prange, \emph{Local energy
  weak solutions for the {N}avier-{S}tokes equations in the half-space}, Comm.
  Math. Phys. \textbf{367} (2019), no.~2, 517--580. \MR{3936125}

\bibitem{MarchandCMP}
Fabien Marchand, \emph{Existence and regularity of weak solutions to the
  quasi-geostrophic equations in the spaces {$L^p$} or {$\dot H^{-1/2}$}},
  Comm. Math. Phys. \textbf{277} (2008), no.~1, 45--67. \MR{2357424}

\bibitem{MarchandPhysicaD}
\bysame, \emph{Weak-strong uniqueness criteria for the critical
  quasi-geostrophic equation}, Phys. D \textbf{237} (2008), no.~10-12,
  1346--1351. \MR{2454593}

\bibitem{PGLR-SQG}
Fabien Marchand and Pierre~Gilles Lemari\'{e}-Rieusset, \emph{Solutions
  auto-similaires non radiales pour l'\'{e}quation quasi-g\'{e}ostrophique
  dissipative critique}, C. R. Math. Acad. Sci. Paris \textbf{341} (2005),
  no.~9, 535--538. \MR{2181389}

\bibitem{miuracriticalsobolevlarge}
Hideyuki Miura, \emph{Dissipative quasi-geostrophic equation for large initial
  data in the critical {S}obolev space}, Comm. Math. Phys. \textbf{267} (2006),
  no.~1, 141--157. \MR{2238907}

\bibitem{novackquasigeostrophic}
Matthew Novack, \emph{Nonuniqueness of weak solutions to the 3 dimensional
  quasi-geostrophic equations}, SIAM J. Math. Anal. \textbf{52} (2020), no.~4,
  3301--3349. \MR{4126319}

\bibitem{Resnick}
Serge~G. Resnick, \emph{Dynamical problems in non-linear advective partial
  differential equations}, 1995, Thesis ({P}h.{D}.)--{T}he {U}niversity of
  {C}hicago, p.~76. \MR{2716577}

\bibitem{SilvestreHigher}
Luis Silvestre, \emph{On the differentiability of the solution to an equation
  with drift and fractional diffusion}, Indiana Univ. Math. J. \textbf{61}
  (2012), no.~2, 557--584. \MR{3043588}

\bibitem{simonforaubinlions}
Jacques Simon, \emph{Compact sets in the space {$L^p(0,T;B)$}}, Ann. Mat. Pura
  Appl. (4) \textbf{146} (1987), 65--96. \MR{916688}

\bibitem{BigStein}
Elias~M. Stein, \emph{Harmonic analysis: real-variable methods, orthogonality,
  and oscillatory integrals}, Princeton Mathematical Series, vol.~43, Princeton
  University Press, Princeton, NJ, 1993, With the assistance of Timothy S.
  Murphy, Monographs in Harmonic Analysis, III. \MR{1232192}

\bibitem{tsai}
Tai-Peng Tsai, \emph{On {L}eray's self-similar solutions of the
  {N}avier-{S}tokes equations satisfying local energy estimates}, Arch.
  Rational Mech. Anal. \textbf{143} (1998), no.~1, 29--51. \MR{1643650}

\bibitem{tsaidiscretely}
\bysame, \emph{Forward discretely self-similar solutions of the
  {N}avier-{S}tokes equations}, Comm. Math. Phys. \textbf{328} (2014), no.~1,
  29--44. \MR{3196979}

\bibitem{wang2019regularity}
Yanqing Wang, Gang Wu, and Daoguo Zhou, \emph{A regularity criterion at one
  scale without pressure for suitable weak solutions to the {N}avier-{S}tokes
  equations}, Journal of Differential Equations \textbf{267} (2019), no.~8,
  4673--4704.

\end{thebibliography}
\bibliographystyle{amsplain}
\end{document}